\documentclass[aop]{imsart}

\RequirePackage[OT1]{fontenc} \RequirePackage{amsthm,amsmath}
\RequirePackage{natbib}
\RequirePackage[colorlinks,citecolor=blue,urlcolor=blue]{hyperref}
\RequirePackage{hypernat}

\usepackage[T1]{fontenc}

\usepackage{float}

\usepackage{amsmath,amsthm,amssymb,amscd}
\usepackage{amsfonts}
\usepackage{enumerate}

\usepackage{verbatim}

\DeclareMathOperator{\var}{Var}

\def\to{\rightarrow}

\def\phi{\varphi}

\def\indic{{\mathbf{1}}}

\numberwithin{equation}{section}
\startlocaldefs
\endlocaldefs

\usepackage{graphicx}
\usepackage{amsmath,amsthm,amssymb,amscd}
\usepackage{amsfonts}
\usepackage{amsmath,amssymb,amsthm,enumerate,epsfig,graphicx}
\usepackage{ifthen,latexsym,syntonly}
\usepackage{natbib}
\usepackage{amsmath}
\usepackage{amssymb}
\newtheorem{thm}{Theorem}

\newtheorem{cor}[thm]{Corollary}

\newtheorem{claim}[thm]{Claim}

\newtheorem{lemma}[thm]{Lemma}
\newtheorem{prop}[thm]{Proposition}

\newtheorem{conj}[thm]{Conjecture}
\theoremstyle{definition}
\newtheorem{remark}[thm]{Remark}
\newtheorem{dfn}[thm]{Definition}

\newcommand{\zz}[1]{\mathbb{#1}}

\newcommand{\Revesz}{R{\'e}v{\'e}sz}

\begin{document}

\begin{frontmatter}
\title{Occupation Statistics of Critical Branching Random Walks
 in Two or Higher Dimensions}
 \runtitle{BRW Occupation Statistics}

\author{\fnms{Steven P.} \snm{Lalley}\ead[label=e1]{lalley@galton.uchicago.edu}\thanksref{t1}}
\thankstext{t1}{Research partially supported by NSF grant DMS  - 0805755}
\address{Department of Statistics, \\The
University of Chicago, \\ Chicago, IL 60637\\ \printead{e1}}
\affiliation{University of Chicago}

\and

\author{\fnms{Xinghua} \snm{Zheng}\ead[label=e2]{xhzheng@ust.hk}  \thanksref{t2}}
\thankstext{t2}{Research partially supported by NSERC (Canada) and the Research Support from Dept of ISOM, HKUST}
\address{Department of Information Systems,\\
\quad Business Statistics and Operations Management \\
Hong Kong University of Science and Technology\\
Clear Water Bay, Kowloon, Hong Kong.\\
\printead{e2}}

\affiliation{Hong Kong University of Science and Technology}
\runauthor{Steven P. Lalley \and Xinghua Zheng}

\begin{abstract}
Consider a critical nearest neighbor branching random walk on the
$d$-dimensional integer lattice initiated by a single particle at
the origin.  Let $G_{n}$ be the event that the branching random walk
survives to generation $n$.  We obtain limit theorems conditional on
the event~$G_{n}$ for a variety of occupation statistics: (1) Let
$V_{n}$ be the maximal number of particles at a single site at time
$n$.  If the offspring distribution has finite $\alpha$th moment for
some integer $\alpha\geq 2$, then in dimensions 3 and higher,
$V_n=O_p(n^{1/\alpha})$; and if the offspring distribution has an
exponentially decaying tail, then $V_n=O_p(\log n)$ in dimensions~3
and higher, and $V_n=O_p((\log n)^2)$ in dimension~2. Furthermore,
if the offspring distribution is non-degenerate then $P(V_n\geq
\delta \log n \,|\, G_{n})\to 1$ for some $\delta >0$.  (2) Let
$M_{n} (j)$ be the number of multiplicity-$j$ sites in the $n$th
generation, that is, sites occupied by exactly $j$ particles. In
dimensions~3 and higher, the random variables $M_{n} (j)/n$ converge
jointly to multiples of an exponential random variable. (3) In
dimension $2$, the number of particles at a ``typical'' site (that
is, at the location of a randomly chosen particle of the $n$th
generation) is of order $O_p(\log n)$, and the number of occupied
sites is $O_p(n/\log n)$. We also show that in dimension 2 there is
particle clustering around a typical site.
\end{abstract}

\begin{keyword}[class=AMS]
\kwd[Primary ]{60J80} \kwd[; secondary ]{60G60} \kwd[; tertiary
]{60F05}
\end{keyword}

\begin{keyword}
\kwd{Critical branching random walks, limit theorems, occupation
statistics}
\end{keyword}

\end{frontmatter}

\section{Introduction} A \emph{nearest neighbor branching random walk}
is a discrete-time particle system on the integer lattice $\mathbb{Z}^d$ that
evolves according to the following rule: At each time
$n=0,1,2,\dotsc$, every particle generates a random number of
offspring, with offspring distribution $\mathcal{Q}=\{Q_l\}_{l\geq
0}$; each of these then moves to a site randomly chosen from among the
$2d+1$ sites at distance $\leq 1$ from the location of the
parent.\footnote{Allowing particles to remain at the same locations as
their parents with positive probability eliminates some annoying
periodicity problems that would require tedious, but routine,
arguments to circumvent.  Our main results could be proved under much
less restrictive hypotheses on the jump distribution.}  We
shall consider only the case where the branching random walk is
\emph{critical}, that is, where the mean number of offspring per
particle is $1$, and we shall assume throughout that the offspring
distribution has finite, positive variance $\sigma^{2}$.

By a well-known theorem of Kolmogorov (see \cite{athreya72},
\hbox{ch.}1) if the branching process is initiated by a single
particle, and if $G_{n}$ is the event that the process survives to
generation~$n$, then
\begin{equation}\label{eq:survivalToGenM}
   \pi_n:=    P (G_{n}) \sim \frac{2}{n\sigma^{2}}.
\end{equation}
Therefore, if the branching random walk is started with $n$
particles at time 0, then the number of initial particles whose
families survive to time $n$ follows, approximately for large $n$,
a Poisson distribution with mean $2/\sigma^2$, and the number of
particles $Z_n$ alive at time $n$ is of order $O_p(n)$.  In fact,
in this case, under suitable hypotheses on the initial
distribution of particles, the measure-valued process associated
with the branching random walk converges, after rescaling, to the
super-Brownian motion $X_{t}$ with variance parameter $\sigma^{2}$
(see \hbox{e.g.}, \cite{etheridge}). In dimensions 2 and higher,
the random measure $X_{t}$ is, for each $t>0$, almost surely
singular with respect to the Lebesgue measure on $\mathbb{R}^d$;
and  when $d \geq 3$, the measure $X_{t}$  spreads its mass over
the support in a fairly uniform manner (\cite{perkins88}), and in
fact can be recovered from its support (\cite{perkins89}). It is
natural to conjecture that this uniformity also holds, in a
suitable sense, for critical branching random walk, and that the
\emph{maximal} number of particles at a single site at time $n$
does not grow rapidly in $n$. Our main results show that this is
indeed the case.  For ease of exposition, we will state our
results as conditional limit theorems given the event $G_{n}$ of
survival to generation $n$. Corresponding unconditional results
for branching random walks started by $n$ particles could easily
be deduced.

 We shall assume throughout the paper, unless otherwise specified,
that the branching random walk is initiated by a single particle
located at the origin at time $0$.  Set
\begin{align}\label{eq:Vm}
             \mathcal{Z}_{n}:&=\text{set of particles in generation }  n;\\
\notag     Z_{n}:&=|\mathcal{Z}_{n}|=\text{number of particles in generation }  n;\\
\notag      U_n(x):&=\text{number of particles at site } x
            \text{ in generation} \; n;\\
\notag  \Omega_{n}:&= \text{number of occupied sites in generation }n;\\
\notag   M_{n} (j):&=\text{number of multiplicity-$j$ sites in
generation }  n; \;\text{and}\\
\notag       V_{n}:&=  \max_{x\in \zz{Z}^{d}} U_{n} (x).
\end{align}
(A \emph{multiplicity-}$j$ \emph{site} is a site with exactly $j$
particles.)

\begin{dfn}\label{definition:Op}
Let  $X_{n} $ be a sequence of random variables, $f (n)$ a sequence
of positive real numbers, and $H_{n}$  a sequence of events of
positive probability. Say that $X_{n} =O_{P} (f (n))$ given $H_{n}$
if the conditional distributions of $X_{n}/f (n)$ given $H_{n}$ are
tight. Similarly, say that $X_{n} =o_{P} (f (n))$ given~$H_{n}$ if
the conditional distributions of $X_{n}/f (n)$ given $H_{n}$
converge weakly to the point mass at $0$.
\end{dfn}

\begin{thm}\label{thm:moment}
Assume that the offspring distribution $\mathcal{Q}$ has  finite
$\alpha$th moment for some integer $\alpha\geq 2$, and that $d\geq
3$. Then conditional on $G_{n}$,
\begin{equation}\label{eq:moment}
    V_{n}=O_{P} (n^{1/\alpha}).
\end{equation}
In particular, if $\mathcal{Q}$ has finite moments of all orders,
then $V_n=o_p(n^\varepsilon)$ for all $\varepsilon >0.$
\end{thm}

\begin{thm}\label{thm:mgf}
Assume that the offspring distribution $\mathcal{Q}$ has an
exponentially decaying tail, that is, there exists $\delta >0$
such that $\sum_l Q_l \exp(\delta l)<~\infty.$ Then conditional
on~$G_n,$
\begin{align}\label{eq:mgf3}
    V_n&=O_p(\log n), \qquad\, \text{if} \; d\geq 3; \quad \\
\label{eq:mgf2}
    V_n&=O_p((\log n)^2),\quad \text{if} \; d=2.
\end{align}
\end{thm}
In fact (see Corollary \ref{cor:mgf} below) for sufficiently large
$C>0$ the conditional probabilities $P(V_n\geq C \log n|G_n)$ in
dimensions $d\geq 3$ and $P(V_n\geq C (\log n)^2|G_n)$ in dimension
$d=2$ decay polynomially in $n$.  For one-dimensional branching
random walk, it is known that $V_{n}$ is of order $\sqrt{n}$ (Theorem
7.10 in \cite{Revesz94}); stronger results are proved in
\cite{lalley07}.

\begin{thm}\label{thm:lu}
Assume that  $d\geq 2$. Then there exists $ \delta>0$, depending
on the offspring distribution~$\mathcal{Q}$, such that
\begin{equation}\label{eq:lu}
 \lim_{n \rightarrow \infty } P(V_n\geq \delta \log n \, | \,G_{n})= 1.
\end{equation}
\end{thm}

Theorem \ref{thm:mgf} and Theorem \ref{thm:lu} imply that, in
dimensions 3 and higher, if the offspring distribution has an
exponentially decaying tail then $V_{n}$ is of order $\log n$ on
the event $G_{n}$ of survival to generation~$n$. In particular,
the (conditional) distributions of ${V_n}/{\log n}$ are tight, and
any weak limit has support contained in $[\delta_{1},\delta_{2}]$
for some $\delta_{1},\delta_{2}>0$ (cf.  Corollary \ref{cor:mgf}).
This partly settles an open question (Question 2, p.79) raised in
\cite{Revesz96}.

\begin{thm}\label{thm:occ3d}
Assume that $d\geq 3$. Then conditional on the event $G_{n}$, the
joint distribution of the occupation statistics $M_{n} (j)/n$
converges as $n \rightarrow \infty$. In particular, for certain
constants $\kappa_{j}$ such that $\sum_{j=1}^{\infty}j\cdot
\kappa_{j}=1$,
\begin{equation}\label{eq:occStats}
    \mathcal{L}\left(\left. \frac{Z_{n}}{n},
    \left\{\frac{M_n(j)}{n}\right\}_{j\geq 1},\frac{\Omega_n}{n}\, \right| \,G_{n}\right)
    \Longrightarrow
    \left(1,\{ \kappa _j\}_{j\geq 1},\sum_j \kappa _j\right)\cdot Y
\end{equation}
where $Y$ is exponentially distributed with mean $2/\sigma^{2}$.
\end{thm}

This extends the classical theorem of Yaglom, according to  which  the
conditional distribution  of $Z_{n}/n$, given that the branching
process survives to generation  $n$, converges to the exponential law
with mean $2/\sigma^{2}$. See \cite{athreya72}, ch.~1 for a discussion
of Yaglom's theorem and related results; and \cite{Geiger00} for an
interesting probabilistic proof.

Theorem \ref{thm:occ3d} implies that in dimensions 3 and higher,
most occupied sites are occupied by only $O (1)$ particles.
Ultimately, this is a consequence of the transience of random walk
in dimensions $d\geq 3$. Since random walk in dimension $d=2$ is
recurrent, different behavior should be expected for the
occupation statistics of branching random walk. In the following
theorem and throughout this article, we shall use the term
\emph{typical particle} to mean a particle chosen randomly from
the $n$th generation $\mathcal{Z}_{n}$ of the branching process
(with the choice made independently of the evolution of the
branching random walk up to time $n$, according to the uniform
distribution on $\mathcal{Z}_{n}$). By a \emph{typical site} we
mean the location of a typical particle.

\begin{thm}\label{thm:ty2d} In dimension $d=2$, the number $T_n$ of
particles at a typical site at time $n$ is, conditional on the
event $G_{n}$, of order $O_p(\log n)$. Moreover, for some
sufficiently small $\varepsilon >0$ there exists $\delta >0$ such
that
\begin{equation}\label{eq:OrderLogn}
    \liminf_{n \rightarrow \infty}
     P (T_{n}\geq \varepsilon \log n \,|\, G_{n})\geq \delta .
\end{equation}
\end{thm}

We conjecture that the conditional distributions of $T_{n}/\log n$
given $G_{n}$ converge in distribution as $n \rightarrow \infty$.
Fleischman \cite{fleischman} has used the method of moments to
establish a related result for the number of particles at a
\emph{fixed} site at distance $O (1)$ from the origin.
Unfortunately, calculation of higher moments for the number of
particles at a \emph{typical} site appears to be considerably more
difficult, and so the method of moments does not seem to be a
feasible approach to the conjecture.

 By Yaglom's theorem, conditional on the event of survival to
generation~$n$ there are $ O_{P} (n)$ particles in all.
Theorem~\ref{thm:ty2d} implies that at least a fraction~$\delta $ of
these are located at sites with (roughly) $ \log n$ other particles.
Thus, a substantial fraction of the particles fall in just
$O_{P}(n/\log n)$ sites.  This does not logically rule out the
possibility that many more sites are occupied; however, it does
suggest that the number $\Omega_{n}$ of occupied sites is of order
$o_{p} (n)$.  This is consistent with the corresponding result for
super-Brownian motion $X_{t}$, which states that for any $t>0$, the
random measure $X_{t}$ is almost surely singular. Following is a
sharp result about the number of occupied sites.

\begin{thm}\label{thm:occ2d}  For
two-dimensional nearest neighbor branching random walk, the number
$\Omega_n$ of occupied sites is $O_p(n/ {\log n })$ given the
event $G_{n}$.
\end{thm}

Theorem \ref{thm:ty2d} implies that the number of occupied sites
must be of order \emph{at least} $n/\log n$. Combining this with
Theorem~\ref{thm:occ2d} we see that $n/\log n$ is the true
asymptotic rate. Revesz \cite{Revesz96} (Theorem~3 (ii)) asserts
that a corresponding result is true for branching Brownian motion,
but we believe that his proof has a serious gap. See
section~\S\ref{ssec:representationUn} for a detailed discussion.

The next theorem partially quantifies the degree of particle
clustering around a typical site.

\begin{thm}\label{thm:ball2d} Assume that $d=2$. Let $\{\ell_n\}$ be
any sequence of real numbers such that $\lim_n\ell_n=\infty$ and
$\lim_n \log \ell_n/\!\log n\!=\!0$. Let $S_{n}$ be the location of a
typical particle, and let $B (S_{n};\ell_n)$ be the ball of radius
$\ell_n$ centered at $S_{n}$. Then conditional on $G_{n}$,
\begin{enumerate}
\item [(A)] the number of \emph{unoccupied} sites in $B
(S_{n};\ell_n)$ is $o_{P} (\ell_n^2)$, and
\item [(B)] the number of particles in $B (S_{n};\ell_n)$ is of order
$O_p(\log n\cdot \ell_n^2)$.
\end{enumerate}
\end{thm}

Theorems~\ref{thm:moment} and \ref{thm:mgf} are proved in section
\S\ref{sec:thm_12}, Theorem~\ref{thm:lu} in section
\S\ref{sec:vn_lu}, and Theorem~\ref{thm:occ3d} in section
\S\ref{sec:occ3d}.  Theorem \ref{thm:ty2d} is proved in section
\S\ref{sec:ty2d}, Theorem~\ref{thm:ball2d} in
section~\S\ref{sec:clustering}, and Theorem \ref{thm:occ2d} in
section \S\ref{sec:occ2d}. For each of the last three theorems the
calculations required for the proofs are considerably simpler in
the special case of \emph{binary fission}, where the offspring
distribution $\mathcal{Q}$ is \emph{double-or-nothing} -- that is,
$Q_{0}=Q_{2}=1/2$. In the interest of clarity, we shall give
complete arguments only for this special case.  These arguments
(as should be evident) can be extended to the general case of mean
$1$, finite variance offspring distributions.

Fundamental to many of our arguments is the following elementary
relation between the expected number of particles at a site $x$ in
generation $n$ and the $n-$step transition probabilities $P_{n} (x)$
of the simple random walk:
\begin{equation}\label{eq:fundamental}
    EU_{n} (x)= P_{n} (x).
\end{equation}
This is easily proved by induction on $n$, by conditioning on the
first generation of the branching random walk.  Here and throughout
the paper, the term \emph{simple random walk} is used for the
symmetric nearest neighbor random walk on the lattice $\zz{Z}^{d}$
{with holding probability} $1/ (2d+1)$ --- that is, each increment
is uniformly distributed on the set $\mathcal{N}$ of $2d+1$ sites at
distance~$\leq 1$ from the origin --- and the notation $P_{n} (x)$
is reserved for the probability that a simple random walk started at
the origin finds its way to site $x$ in $n$ steps. We use the
notation $\zz{P}^{n}$ to denote the $n-$step transition probability
kernel of simple random walk, that is, the $n$th iterate of the
Markov operator $\zz{P}:\ell^{\infty} (\zz{Z}^{d}) \rightarrow
\ell^{\infty} (\zz{Z}^{d})$ associated with the random walk.

\medskip \noindent \textbf{Notation.}  Following is a list of
notation, in addition to that already established in
equations~\eqref{eq:survivalToGenM}, (\ref{eq:Vm}) and
\eqref{eq:fundamental} above, that will be fixed throughout the paper:

\begin{itemize}
 \item $\mathcal{N}=\{e_j\}_{-d\leq j\leq d} $ is the set of
sites at distance $0$ or $1$ from the origin  in $\mathbb{Z}^d$.
\item $\mathcal{Q}=\{Q_{l}\}_{l\geq 0}$ is the offspring
distribution, and $\mathcal{Q}^i=\{Q^{i}_{l}\}$ its $i$th
convolution power.
 \item $\mathcal{F}_{n}$ is the $\sigma -$algebra generated
by the random variables $\{U_{m} (x) \}_{x\in \zz{Z}^{d}, m\leq
n}$.
 \item $A=5/(4\pi)$ is the  constant such that
$P_n(0)\sim A/n$ in dimension~2, see, e.g., P7.9 on Page~75
in~\cite{spitzer76} .
\end{itemize}

In addition, we will follow the custom of writing $f\sim g$ to
mean that the ratio $f/g$ converges to 1, and $f\asymp g$ to mean
that the ratio $f/g$ remains bounded away from $0$ and $\infty$.
Throughout the paper, $C,C_1, C'$ \hbox{etc.} denote generic
constants whose values may change from line to line. Finally, we
use a ``local scoping rule'' for notation: Any notation introduced
in a proof is local to the proof, unless otherwise indicated.

\section{Proofs of Theorem \ref{thm:moment} and
\ref{thm:mgf}}\label{sec:thm_12}
\subsection{The case where the offspring distribution has finite
moments} The proof of Theorem~\ref{thm:moment} will rely on the
following estimates for  the moments of the occupation statistics
$U_{n} (x)$.

\begin{prop}\label{prop:moment}Suppose that the offspring distribution
$\mathcal{Q}$ has finite $\alpha$th moment for some integer
$\alpha\geq 2$.
\begin{itemize}
 \item[(i)] If $d\geq 3,$ then
$$
  \sup_n\sum_x EU_n(x)^\alpha<\infty.
$$
 \item[(ii)] If $d=2,$ then there exist $C_1, C_2<\infty$  such that for all
$n$,
$$
  \sum_x EU_n(x)^\alpha\leq C_1 n^{C_22^\alpha}.
$$
\end{itemize}
\end{prop}
\begin{proof}
We will use the following inequality: For all $l\geq 2$ and all
$b_i\geq 0,$
\begin{equation}\label{eqn:sum_moment}
\aligned
  \left(\sum_{i=1}^l b_i\right)^{\!\alpha}
   &\leq \sum_{k=2}^\alpha\sum_{\mathcal{P}_k}\left(\sum_{\ell=1}^k b_{i_\ell}\cdot
    \indic_{\{b_{i_1}>0,\ldots,b_{i_k}>0\}}\right)^\alpha,
\endaligned
\end{equation}
where $\mathcal{P}_k$ is the set of $k-$tuples $(i_1,\ldots,i_k)$ of
distinct positive integers no greater than $l$.
Inequality~\eqref{eqn:sum_moment} is obviously true for $l\leq
\alpha$. To see that it holds for $l>\alpha$, observe that, by the
multinomial expansion, the left side of \eqref{eqn:sum_moment} is a
sum of terms of the form $t={\alpha\choose j_1\;j_2\ldots j_l}
b_1^{j_1}b_2^{j_2}\ldots b_l^{j_l}$, where the exponents~$j_{i}$ sum
to $\alpha$. Since at most  $ \alpha$ of these can be positive, and
$t$ vanishes if any of $b_{i}$ with exponent $j_i>0$ is zero, the
term $t$ is included in the sum on the right side of
\eqref{eqn:sum_moment}.

Next, by the H\"{o}lder inequality, for each integer $k\geq 2$ and
all real numbers $b_i\geq 0 $,
\begin{equation} \label{eqn:holder}
   \left(\sum_{i=1}^k b_i\right)^{\!\alpha}\leq
   k^{\alpha-1}\sum_{i=1}^k b_i^\alpha.
\end{equation}
This implies that if $k$  independent branching random walks are
started by particles $u_1,\ldots, u_k$ located at sites
$x_1,\ldots,x_k$ respectively, and if $U_n^{u_i}(x)$ is the number
of the $n$th generation descendants at site $x$ of the particle
$u_i$, then
\begin{equation}\label{eqn:moment_indc}
\aligned
   &\sum_x E\left(\sum_{i=1}^k U_n^{u_i}(x)\cdot
          \indic_{\{U_n^{u_1}(x)>0,\ldots,
         U_n^{u_k}(x)>0\}}\right)^{\!\alpha}\\
   \leq& k^{\alpha-1}\sum_{i=1}^k\sum_x E (U_n^{u_i}(x))^\alpha\cdot
          \prod_{j\neq
         i}P(U_n^{u_j}(x)>0)\\
   \leq& k^\alpha\sum_x EU_n(x)^\alpha\cdot \left(C\frac{1}{\sqrt{n}^d}\right)^{\!k-1}.
\endaligned
\end{equation}
Here we have used (\ref{eqn:holder}) in the first inequality; the
second inequality follows by the local central limit theorem and the
elementary observation that
$$
   P(U_n^{u_j}(x)>0)\leq EU_n^{u_j}(x)=P_n(x-x_j).
$$

We are now prepared to estimate $\sum_x EU_{n}(x)^\alpha$.
Conditioning on the first generation,  we obtain
$$\aligned
   &\sum_x EU_{n}(x)^\alpha\\
   \leq& \sum_x EU_{n-1}(x)^\alpha
   +\sum_{k=2}^\alpha\sum_x
    E\left[\sum_{\mathcal{P}_k}\left(\sum_{j=1}^k U_{n-1}^{u_j}(x) \cdot
      \indic_{\{U_{n-1}^{u_1}(x)>0,\ldots,
    U_{n-1}^{u_k}(x)>0\}}\right)^{\!\alpha}\right]\\
  \leq& \sum_x EU_{n-1}(x)^\alpha\cdot\left(1+ \sum_{k=2}^\alpha\sum_l Q_l{l\choose k} k^{\alpha}\cdot
       \left(C\frac{1}{\sqrt{n-1}^d}\right)^{\!k-1}\right),
\endaligned$$
where $\mathcal{P}_k$ denotes the set of $k-$tuples $(
u_1,\ldots,u_k)$ of distinct particles in generation~1, and the
first and second inequality hold by \eqref{eqn:sum_moment}
and \eqref{eqn:moment_indc} respectively. Therefore, for all~$n,$
$$
    \sum_x EU_{n}(x)^\alpha
    \leq \prod_{i=2}^n \left(1+ \sum_{k=2}^\alpha\sum_l Q_l{l\choose k}
         k^{\alpha}\cdot \left(C\frac{1}{\sqrt{i-1}^d}\right)^{\!k-1}\right)\cdot
         \sum_x  EU_{1}(x)^\alpha.
$$
Clearly, $\sum_x  EU_{1}(x)^\alpha\leq (2d+1)EZ_1^\alpha<\infty$.
Furthermore,  in dimensions $d\geq 3$,
$$\aligned
   &\prod_{i=2}^n \left(1+ \sum_{k=2}^\alpha\sum_l Q_l{l\choose k}
         k^{\alpha}\cdot \left(C\frac{1}{\sqrt{i-1}^d}\right)^{\!k-1}\right)\\
   \leq& \exp\left(\sum_{i=2}^\infty \sum_{k=2}^\alpha\sum_l
      Q_l{l\choose k} k^{\alpha}\cdot
      \left(C\frac{1}{\sqrt{i-1}^d}\right)^{\!k-1}\right)\\
   =&\exp\left(C'\sum_{k=2}^\alpha\sum_l Q_l{l\choose k} k^{\alpha}\right),
\endaligned$$
where $C'<\infty$ is  independent of $n$; and  in dimension
$d=2$,
$$\aligned
   &\prod_{i=2}^n \left(1+ \sum_{k=2}^\alpha\sum_l Q_l{l\choose k}
         k^{\alpha}\cdot
         \left(\frac{C}{i-1}\right)^{\!k-1}\right)\\
   \leq& \exp\left(C\sum_l Q_l{l\choose 2} 2^{\alpha} \cdot
     \sum_{i=2}^n \frac{1}{i-1} +C\sum_{k=3}^\alpha\sum_l Q_l{l\choose k} k^{\alpha}\cdot
     \sum_{i=2}^\infty\left(\frac{1}{i-1}\right)^{ k-1}\right)\\
   \leq& \exp(C_22^\alpha\log n + C_3),
\endaligned$$
where $C_2$ is a constant independent of both $\alpha$ and $n$,
and $C_3$ is a constant independent of~$n$.
\end{proof}

\begin{proof}[Proof of Theorem \ref{thm:moment} ] By Kolmogorov's
estimate \eqref{eq:survivalToGenM}, the probability that the process
survives to time $n$ is $O(1/n)$.  By the Markov inequality,
\[
    P\{V_{n}\geq Cn^{1/\alpha} \} \leq
    C^{-\alpha}n^{-1}EV_{n}^{\alpha}
    \leq C^{-\alpha} n^{-1}E\sum_{x} U_{n} (x)^{\alpha},
\]
and so the relation \eqref{eq:moment} follows from Proposition
\ref{prop:moment}.
\end{proof}

\begin{remark}
Yaglom's limit theorem implies that, conditional on the event
$G_{n}$, the number of particles at time $n-1$ is $O_p(n)$. For
each of these, there is a small chance that the number of
offspring will exceed $(2d+1) n^{1/ (\alpha +\varepsilon )}$, in
which case $V_{n}$ will be at least $n^{1/ (\alpha +\varepsilon
)}$. If the tail of the offspring distribution decays like
$m^{-(\alpha+\varepsilon)}$ as $m\to\infty$, then the chance that
one of the $O_{p} (n)$ particles in generation $n-1$ will have
more than $(2d+1) n^{1/ (\alpha +\varepsilon )}$ offspring is of
order one. Thus, the result in Theorem \ref{thm:moment} is almost
optimal. (This answers a question of Michael Stein.)
\end{remark}

\subsection{The case where the offspring distribution has an
exponentially decaying
tail}\label{ssec:exponentialTail}

We begin with a stochastic comparison result for the random variables
$U_{n} (x)$.  First, observe that the law of the branching random walk
(started by a single particle located at the origin) is invariant with
respect to reflections in the coordinate axes, and so $U_{n}
(x)\stackrel{\mathcal{D}}{=}U_{n} (x')$ for any two sites $x,x'$ at
corresponding positions of different orthants. Now
define the usual partial order on the positive orthant
$\mathbb{Z}^d_{+}$:
\[
    x \preceq y \quad \text{if} \quad x_{i}\leq y_{i} \;
    \text{for all} \;1\leq i\leq d.
\]

\begin{lemma}\label{lemma:um0}
If $x\preceq y$ then $U_{n} (x)$ stochastically dominates
$U_{n}(y)$; in particular, $U_{n} (y)$ is stochastically dominated
by $U_{n} (0)$ for every $y\in \zz{Z}^{d}$. Consequently, if
$x\preceq y$, then for every $n\geq 0$,
\begin{align}\label{eq:comparisonA}
    P_{n} (x)&\geq P_{n} (y) \quad \text{and}\\
\label{eq:comparisonB}
    u_{n} (x)&\geq u_{n} (y),
\end{align}
where $u_{n} (x):=P\{U_{n} (x)\geq 1 \}$ is the \emph{hitting
probability function} of the branching random walk.
\end{lemma}

\begin{remark}\label{remark:reflection}
The relation \eqref{eq:comparisonA}, which follows from the stochastic
dominance $U_{n} (x)\geq_{\mathcal{D}}U_{n} (y)$ by taking
expectations (recall the fundamental relation \eqref{eq:fundamental}),
also follows more directly by the reflection principle for simple
random walk.
\end{remark}

\begin{proof} Because the law of the branching random walk is
invariant with respect to permutations of the coordinates, we may
assume, without loss of generality,  that $y=x+e_{1}$, where $e_{1}=
(1,0,\dotsc ,0)$.
Denote
by $L$ and $L'$ the hyperplanes
\begin{align*}
    L&=\{z\in \zz{R}^{d}\, : \,z_{1}=x_{1}\} \quad \text{and}\\
    L'&=\{ z\in \zz{R}^{d}\, : \,z_{1}=x_{1}+1/2\};
\end{align*}
observe that $y$ is the reflection of $x$ in $L'$.  We shall define a
particle system with particles of three colors --- red, blue, and
green --- in such a way that
\begin{itemize}
   \item [(a)] the subpopulation of all red and blue particles
         follows the law of the branching random walk  started by one (red)
       particle at the origin;
   \item [(b)] the subpopulation of all red and green particles follows the same law;
   \item [(c)]  there are no red particles to the right of the hyperplane $L'$; \quad and
   \item [(d)] at each time, the green and blue particles are
   paired (bijectively) in such a way that the green and blue particles in
   any pair are at symmetric locations on opposite sides of the hyperplane~$L'$.
\end{itemize}
This will prove that $U_{n} (x)\geq_{\mathcal{D}}U_{n} (y)$ for
each $n$, by the following reasoning: First, the distribution of
$U_{n} (x)$ coincides with the distribution of the total number of
\emph{red} and \emph{blue} particles at location $x$ and time $n$,
by (a). Second, the number of \emph{blue} particles at $x$ equals
the number of \emph{green} particles at~$y$, by~(d), since $x$ and
$y$ are at symmetric locations on opposite sides of the
hyperplane~$L'$. Third, the number of \emph{green} particles
at~$y$ has the same distribution as $U_{n} (y)$,  by  (b) and (c).

The particle system is constructed as follows.  To start, color the
initial particle at the origin red. Offspring of blue and green
particles will always have the same color as their parents, and each
blue particle $b$ will always be paired with a green particle $g$
located at the mirror image (relative to reflection in the
hyperplane $L'$) of the site of $b$. Offspring of red particles will
be red except possibly when the parent red particle is located at a
site on the hyperplane $L$. In this case --- say, for definiteness,
that the red parent particle~$\xi$ is at site $z\in L$ --- each
offspring particle $\zeta$ first makes a jump according to the law
of the nearest neighbor random walk, and then chooses a color as
follows: (a) If the jump is to a site $z'\not =z$ to the \emph{left}
of hyperplane~$L'$ then $\zeta $ becomes \emph{red}; and (b) If the
jump is either to the same site $z$ as the parent or to its mirror
image $z^{*}$ on the \emph{right} of $L'$ then $\zeta $ chooses
randomly between blue and green. In case (b) the offspring
particle~$\zeta$ generates a \emph{doppelganger} (mirror particle)
$\zeta '$ of the opposite color at the reflected site on the other
side of~$L'$. Note the distribution of the position of $\zeta$ is
the same as that of $\zeta'$. The particle~$\zeta$ generates an
offspring branching random walk $\mathcal{G}_{\zeta} $ with all
particles having the same color as $\zeta $; the mirror image
$\mathcal{G}_{\zeta '}$ of $\mathcal{G}_{\zeta}$ relative to $L'$
(with particles colored oppositely) is attached to $\zeta '$. Note
that $\mathcal{G}_{\zeta '}$ is itself a branching random walk
started at the location of $\zeta '$, by the symmetry of the nearest
neighbor random walk.

Properties (a)--(d) above are now readily apparent. Property (c)
holds because, by construction, children of red particles on $L$
that jump across $L'$ are either green or blue, and offspring of
blue and green particles are either blue or green. Property~(d) is
inherent in the construction. Finally, (a) and (b) follow from the
blue/green symmetry of the reproduction law for red particles
located at sites on~$L$.
\end{proof}

\begin{prop} \label{prop:mgf} Assume that the offspring distribution
$\mathcal{Q}$ has finite moment generating function in some
neighborhood of the origin. Then in dimensions $d\geq 3,$ there
exist $\delta_d>0$ and $C>0$ such that for any $\theta \in
[0,\delta_d]$, all $x\in\zz{Z}^d$ and all $n\geq 1$,
\begin{equation}\label{eq:mgfD3}
    E \exp \{\theta U_{n} (x) \}-1 \leq CP_{n} (x)\theta .
\end{equation}
In dimension $d=2$, there exist $\delta_2>0$ and  $C>0$ such that
for any $\theta \in [0,\delta_2]$, all $x\in\zz{Z}^2$ and all
$n\geq 1$,
\begin{equation} \label{eqn:2d_mgf}
         E \exp \{\theta U_{n} (x) /\log n\}-1 \leq CP_{n} (x)\theta /\log n.
\end{equation}
\end{prop}

\begin{proof}
Let $\Phi (z)=\sum_{l=1}^{\infty}Q_{l}z^{l}$ be the probability
generating function of~$\mathcal{Q}$. By hypothesis, $\Phi (z)$ is
finite and analytic in a neighborhood of the closed disc $|z|\leq
e^{\delta}$ for some $\delta >0$, and since the variance of
$\mathcal{Q}$ is strictly positive, $\Phi (z)$ is strictly convex on
$[0,e^{\delta}]$. Moreover, $\Phi ' (1)=1$,  because the offspring
distribution has mean $1$.

Define
\begin{equation*}
    G_{n} (x)=G_{n} (x;\theta) =E\exp(\theta U_{n}(x))-1.
\end{equation*}
Clearly, $G_{n} (x;\theta) \rightarrow 0$ as $\theta \rightarrow
0$. Moreover, by Lemma~\ref{lemma:um0}, for each value of $\theta
>~0$ the function $G_{n} (x)$ is maximal at $x=0$.  Since the
random variables $U_{1} (x)$ are zero except for $x\in~\mathcal{N}$,
and have the same distribution for $x\in~\mathcal{N}$, the function
$G_{1} (x)$ is, for any fixed $\theta$, a scalar multiple of the
uniform distribution $P_{1}$ on $\mathcal{N}$. Conditioning on the
first generation of the branching random walk shows that
\begin{equation}\label{eq:gmMarkov}
      G_{n+1}(x)+1=\Phi (\zz{P}G_{n} (x)+1)
\end{equation}
where $\zz{P}$ is the one-step Markov operator for the simple
random walk, that is, $\zz{P}f (x)=Ef (x+Y)$ where $Y$ is
uniformly distributed on $\mathcal{N}$.  Since $\Phi(1)=1$ and
$\Phi (z)$ is strictly convex for $z\in [0,e^{\delta}]$, equation
\eqref{eq:gmMarkov} implies that
\begin{equation}\label{eq:gm2}
        G_{n+1}(x) \leq \zz{P}G_{n} (x) \Phi '(1+\zz{P}G_{n} (x)).
\end{equation}

Unfortunately, both relations \eqref{eq:gmMarkov} and \eqref{eq:gm2}
are nonlinear in $G_{n}$. For this reason, we introduce dominating functions
$H_{n} (x)=H_{n} (x;\theta )$ that satisfy corresponding \emph{linear}
relations: Set $H_{1} (x)=G_{1} (x)$, and define $H_{n}$ inductively
by
\begin{equation}\label{eqn:hminduction}
    H_{n+1}(x) = \zz{P}H_{n}(x) \Phi ' (1+H_{n} (0)).
\end{equation}
Note that $H_{n+1}$ may take the value $+\infty$ if $H_{n} (0)$
exceeds the radius of convergence of $\Phi$.
Since $G_{n} (x)\leq G_{n} (0)$, the inequality \eqref{eq:gm2} implies
that $H_{2}\geq G_{2}$, and so by induction that $H_{n}\geq G_{n}$ for
all $n\geq 1$. Thus, to prove inequalities \eqref{eq:mgfD3} and
\eqref{eqn:2d_mgf} it suffices to prove analogous inequalities for the
functions $H_n(x;\theta )$.

The advantage of working with the functions $H_{n}$ is that the linear
relation \eqref{eqn:hminduction} can be iterated. In general, if
functions $f$ and $g$ satisfy $g=a\zz{P}f$ for some scalar $a$, then
$\zz{P}g=a\zz{P}^{2}f$. Employing this identity in equation
\eqref{eqn:hminduction} and iterating yields
\[
    H_{n} (x)=\zz{P}^{n-1}H_{1} (x)\prod_{j=1}^{n-1} \Phi ' (1+H_{j} (0)).
\]
Because the function $H_{1}=G_{1}$ is itself a scalar multiple of
$P_{1}$, it follows that
\begin{equation}\label{eq:HnIterative}
     H_{n}(x;\theta )=P_{n}(x)H_{1} (0;\theta ) (2d+1)
                \prod_{j=1}^{n-1} \Phi' (1+H_{j} (0;\theta )).
\end{equation}
Since $\Phi ' (1)=1$, the factors in the product are well-approximated
by $(1+\Phi '' (1)H_{j} (0;\theta))$ as long as $H_{j} (0;\theta)$
remains small. In particular, for suitable constants $C<\infty$ and
$\varepsilon >0$, if $H_{j} (0;\theta)<\varepsilon$ for all
$j\leq n-1$ then
\begin{equation}\label{eq:hProdIneq}
            H_{n}(x;\theta )\leq  (2d+1) P_{n}(x)H_{1} (0;\theta )
                   \prod_{j=1}^{n-1} (1+CH_{j} (0;\theta)),
\end{equation}
equivalently,
\begin{equation}\label{eq:hFracIneq}
    \frac{H_{n} (0;\theta )}{\prod_{j=1}^{n} (1+CH_{j}
    (0;\theta))}
    \leq (2d+1) P_{n} (0) H_{1} (0;\theta ) .
\end{equation}
The large-$n$ behavior of the products on the right side of
\eqref{eq:hProdIneq} will depend on whether or
not the sequence $P_{n} (0)$ is summable, that is, on whether or not
the simple random walk is transient. There are two cases to consider:

\medskip \noindent \textbf{Dimensions $d\geq 3$:} In dimensions $d\geq
3$, the return probabilities $P_{n} (0)$ are summable. Moreover, when
$\theta>0$ is small, the factor $(2d+1)H_{1} (0;\theta)$ on the right
side of \eqref{eq:hFracIneq} is also small, because $H_{1}=G_{1}$ is a
continuous function of $\theta$ that takes the value $0$ at $\theta
=0$.  Hence, by choosing $\theta$ small we can make the sum over $n$
of the quantities on the right side of inequality \eqref{eq:hFracIneq}
arbitrarily small. Now the fraction on the left side of
\eqref{eq:hFracIneq} is the $n$th term of the telescoping series
\begin{equation}\label{eq:telescoping}
    C^{-1}\sum \left( \frac{1}{\prod_{j=1}^{n-1} (1+CH_{j}(0;\theta))}
    -\frac{1}{\prod_{j=1}^{n} (1+CH_{j}(0;\theta))} \right);
\end{equation}
consequently,  \eqref{eq:hFracIneq} implies that
for all sufficiently small $\theta >0$ the products
\[
    \prod_{j=1}^{n} (1+CH_{j}(0;\theta))
\]
remain bounded for large $n$, and for small $\theta$ remain close
to $0$. It now follows by \eqref{eq:hProdIneq} that for a suitable
constant $C'<\infty$ and all small $\theta$ the functions $H_{n}
(x;\theta)$ are all finite, and satisfy
\[
        H_{n}(x;\theta )\leq C' P_{n}(x)H_{1} (0;\theta ).
\]
Finally, the differentiability of $H_{1} (0;\theta)$ in $\theta$
guarantees that $H_{1} (0;\theta)\leq C\theta$ for an appropriate
constant $C<\infty$ for all small $\theta$.  This proves
\eqref{eq:mgfD3}.

\medskip \noindent \textbf{Dimension $d=2$:} It is still the case that
the fraction on the left side of \eqref{eq:hFracIneq} is the $n$th
term of the telescoping series \eqref{eq:telescoping}, but since
$\sum P_{n} (0)$ diverges, this no longer implies that the
products on the right side of \eqref{eq:hProdIneq} remain bounded.
However, the local central limit theorem gives an explicit
estimate for the partial sums of the return probabilities: in
particular, for some $C'\geq A=5/ (4\pi)$,
\[
    \sum_{j=1}^{n}P_{j} (0)\leq  C\log n
    \quad \text{for all} \;n\geq 2.
\]
Consequently, substituting $\theta /\log n$ for $\theta$ in inequality
\eqref{eq:hFracIneq} and summing gives
\[
    1-\prod_{j=1}^{n} (1+CH_{j}(0;\theta/\log n))^{-1}
    \leq C'' \theta .
\]
This in turn implies that
\[
    \prod_{j=1}^{n} (1+CH_{j}(0;\theta/\log n))
    \leq 1/ (1-C''\theta).
\]
Using this upper bound for the product on the right side of
\eqref{eq:hProdIneq} and using the bound $H_{1} (0;\theta/\log
n)\leq C\theta/\log n$ for small $\theta$ yields
\eqref{eqn:2d_mgf}.
\end{proof}

\begin{remark}
In dimensions $d\geq 3$, the conclusion \eqref{eq:mgfD3} cannot be
extended to all $\theta >0$,  even for the double-or-nothing case.
In fact, for sufficiently large $\theta $, the sums $\sum_{x\in
\zz{Z}^{d}}\left(E\exp \{\theta U_{n} (x) \}-1\right)$ are  not
 bounded in~$n$.
\end{remark}

\begin{remark}
In dimension $d=2$, the relation (\ref{eqn:2d_mgf})  does not hold for large
$\theta $. See Remark \ref{rmk:2nd_moment} below.
\end{remark}

We are now prepared to prove Theorem \ref{thm:mgf}. In fact, we will establish
the following  stronger result:

\begin{cor}\label{cor:mgf}  Under the hypotheses of Proposition \ref{prop:mgf},
with the same notations,
\begin{itemize}
 \item[(i)] If $d\geq 3,$ then for all $\theta \leq \delta_{d}$,
$$
   P\left(\left.V_n\geq  \frac{\log n}{\theta}\right|G_n\right)
   =O\left(\frac{1}{n^{\delta_d/\theta-1}}\right).
$$
In particular, conditional on $G_n$,
$V_n=O_p(\log n)$.\\
 \item[(ii)] If $d=2,$ then for all $\theta \leq \delta_{2}$,
$$
   P\left(\left.V_n\geq  \frac{(\log n)^2}{\theta}\right|G_n\right)
   =O\left(\frac{1}{n^{\delta_2/\theta}-1}\right).
$$
In particular, conditional on $G_n$,
$V_n=O_p((\log n)^2)$.\\
\end{itemize}
\end{cor}

\begin{proof}
We will prove this  only for dimensions $d\geq 3$; the
dimension $d=2$ case can be handled similarly. By Markov's inequality,
$$\aligned
   P\left(\left.V_n\geq  \frac{\log n}{\theta}\right|G_n\right)
   &\leq \frac{1}{\exp(\delta_d/\theta \cdot \log n)}\sum_x E \left(\left.
      \exp(\delta_d U_n(x))\cdot \indic_{\{U_n(x)> 0\}} \right| G_n\right)\\
   &=O\left( \frac{1}{n^{\delta_d/\theta}-1}\cdot\sum_x E\left(\exp(\delta_d U_n(x))
    \cdot \indic_{\{U_n(x)> 0\}}\right)\right).
\endaligned$$
For any random variable $X\geq 0$,
$$\aligned
   E \exp(X)
   &=E \exp(X)\cdot \indic_{\{X>0\}} + E \exp(X)\cdot \indic_{\{X=0\}}\\
   &=E \exp(X)\cdot \indic_{\{X>0\}} + P(X=0)\\
   &=E \exp(X)\cdot \indic_{\{X>0\}} + 1 - P(X>0).
\endaligned$$
Hence
$$\aligned
  &\sum_x E \left(\exp(\delta_d U_n(x))\cdot \indic_{\{U_n(x)>
  0\}}\right)\\
   =&\sum_x \left(E \exp({\delta_d U_n(x)}) -1 \right)+ \sum_x P(U_n(x)>0)\\
  \leq& \sum_x \left(E \exp({\delta_d U_n(x)}) -1\right) + 1,
  \endaligned
$$
so by Proposition \ref{prop:mgf},
$$
  \sum_x E \left(\exp({\delta_d U_n(x)})\cdot \indic_{\{U_n(x)> 0\}}\right)\leq C
  \quad \text{for all}\; n\geq 1.
$$
The conclusion follows.
\end{proof}

\section{Proof of Theorem  \ref{thm:lu}}\label{sec:vn_lu}

The proof uses the following elementary lemma, whose proof is left to
the reader.

\begin{lemma}\label{lemma:conditioning} Suppose that on some probability
space $(\Omega, \mathcal{F}, P)$ there are two events $E_1, E_2$
such that
\begin{equation}\label{eqn:sym_diff}
    \frac{P(E_1\Delta E_2)}{P(E_{1})}\leq \varepsilon,
\end{equation}
where $E_1\Delta E_2$ is the symmetric difference of $E_1$ and
$E_2$. Then
\begin{equation}\label{eqn:tv_diff}
   ||P(\cdot|E_1) - P(\cdot|E_2)||_{TV}
   \leq 2\varepsilon,
\end{equation}
where $P(\cdot|E_i)$ denotes the conditional probability measure given
the event $E_{i}$ and $||\cdot||_{TV}$ denotes the total variation distance.
\end{lemma}

Lemma \ref{lemma:conditioning} will allow us to replace the event of
conditioning $G_{n}$ in Theorems~\ref{thm:lu} and \ref{thm:occ3d} by
asymptotically equivalent events of the form
\begin{equation}\label{eq:hn}
    H_{n}=\{Z_{m(n)}\geq n\varepsilon_{n} \}.
\end{equation}

\begin{lemma}\label{lemma:timeJiggle}
Let $m(n)< n$ be  integers and $\varepsilon_{n}>0$   real numbers
such that  $m(n)/n \rightarrow 1$ and $\varepsilon_{n} \rightarrow
0$ as $n \rightarrow \infty$. Then
\begin{equation}\label{eq:noAsyDiff}
    \lim_{n \rightarrow \infty} \frac{P (G_{n}\Delta H_{n})}{P(G_{n})}=0.
\end{equation}
\end{lemma}

\begin{proof}
This is an easy consequence of Kolmogorov's estimate
\eqref{eq:survivalToGenM} and Yaglom's theorem for critical
Galton-Watson processes. Let $K_{n}=\{Z_{m(n)}\geq~1 \}$. Clearly,
$H_{n}\subset K_{n}$, and so $P (K_{n}\,|\,H_{n})=~1$. On the other
hand, Yaglom's theorem implies that $P (H_{n}\,|\,K_{n})\rightarrow
1$, since $m(n)/n \rightarrow 1$. Consequently,
\begin{equation}\label{eq:equiv_HK}
    \lim_{n \rightarrow \infty}\frac{P (H_{n}\Delta K_{n})}{P (K_{n})}=0.
\end{equation}

A similar argument shows that the symmetric difference $K_{n}\Delta
G_{n}$ is an asymptotically negligible part of $K_{n}$. Obviously,
$G_{n}\subset K_{n}$, so $P (K_{n}\,|\,G_{n})=~1$. Yaglom's theorem
implies that for any $\delta
>0$ there exists $\alpha >0$ such that
\[
    P (Z_{m(n)}>\alpha n \,|\, K_{n})\geq 1-\delta.
\]
But on the event $\{Z_{m(n)}>\alpha n\} $ the event $G_{n}$ of
survival to generation $n$ is nearly certain for large~$n$,
because  the $Z_{m(n)}$ particles in generation $m(n)$ initiate
independent Galton-Watson processes, each of which survives to
generation $n$ with probability $\sim 2/ (n-m(n))\sigma^{2}$, by
Kolmogorov's estimate~\eqref{eq:survivalToGenM}.  Hence,
\[
    P (G_{n}\,|\, K_{n})\geq 1-2\delta
\]
for large $n$. Since $\delta >0$ is arbitrary, it follows that $P
(G_{n}\,|\,K_n) \rightarrow 1$. By Lemma \ref{lemma:conditioning}
and \eqref{eq:equiv_HK} we get
\begin{equation}\label{eq:GH}
   P(G_n\, |\, H_n)\to 1.
\end{equation}
Furthermore, since $G_{n}\subset K_{n}$,
\[
    \lim_{n \rightarrow \infty}\frac{P (G_{n}\Delta K_{n})}{P (K_{n})}=0.
\]
By Lemma~\ref{lemma:conditioning}, this implies that conditioning on
$G_{n}$ is asymptotically equivalent to conditioning on $K_{n}$, and
so the difference $P (H_{n}\,|\,K_{n})-P (H_{n}\,|\,G_{n})
\rightarrow~0$. But we have seen that $P (H_{n}\,|\,K_{n})
\rightarrow~1$, hence $P(H_n\,|\,G_n)\to 1.$ This, along with
\eqref{eq:GH}, implies \eqref{eq:noAsyDiff}.
\end{proof}

\begin{proof}[Proof of Theorem \ref{thm:lu}]
The offspring distribution is non-degenerate, so there exists
$l_0>1$ such that $Q_{l_0}>0$. Let $p=Q_{l_0}\cdot
\left(1/(2d+1)\right)^{l_0}$ be the probability that the initial
particle produces $l_0$ offspring  and these offspring all stay at
the origin. Then for all $k\in \mathbb{N}$,
$$
 P(U_k(0)\geq l_0^k) \geq p\cdot p^{l_0}\cdot p^{l_0^2}\ldots
 p^{l_0^{k-1}} \geq p^{l_0^k/(l_0-1)}.
$$
Our objective is to show that for some $\delta>0$, $P(V_n\geq \delta
\log n|Z_n>0)\to 1$. By Lemmas \ref{lemma:conditioning} and
\ref{lemma:timeJiggle}, this will follow if we can show that for some
$m(n)\leq n$ with $m(n)/n\to 1$ and some $\varepsilon_n\to 0$, the
probability
$$
   P(V_n\geq \delta \log n|Z_{m(n)}>n\varepsilon_n)\to 1.
$$
To do so, for $\delta>0$ to be determined later, and all $n$ big
enough, define $k$ such that $l_0\delta \log n > l_0^k \geq \delta
\log n,$ and $m(n)=n-k$. Then $m(n)/n\to 1$.  Fix a
sequence $\varepsilon_n=O(1/\log n)$; then
\begin{equation}\label{eqn:vn_gn}
\aligned
   &P\left(V_n\geq \varepsilon \log n \left| \frac{Z_{m(n)}}{n} \geq
         \varepsilon_n\right.\right)\\
   \geq& 1-\left(1-P(U_k(0)\geq l_0^k)\right)^{\varepsilon_n n}\\
   \geq& 1-\left(1-p^{l_0^k/(l_0-1)}\right)^{\varepsilon_n n}\\
   \geq& 1-\exp\left(\varepsilon_n n \left(-p^{l_0^k/(l_0-1)}\right)\right)\\
   \geq& 1-\exp\left(-\varepsilon_n n p^{l_0\delta \log n/(l_0-1)}\right)\\
   =&1-\exp\left(-\varepsilon_n n^{1+l_0\delta\log
         p/(l_0-1)}\right)
   \to 1
\endaligned
\end{equation}
provided that $\delta <(l_0-1)/(-l_0\log p).$
\end{proof}

\section{Proof of Theorem \ref{thm:occ3d} }\label{sec:occ3d}

\subsection{Strategy}\label{ssec:strategyTh4} By Lemmas
\ref{lemma:conditioning} and \ref{lemma:timeJiggle}, the difference
between conditioning on the event $G_n=\{Z_n>0\}$ and conditioning on
the event $H_n:=\{Z_{m(n)}\geq n\varepsilon_n\}$ is
asymptotically negligible if $m(n)/n \rightarrow 1$ and
$\varepsilon_{n} \rightarrow 0$. Thus, it suffices to prove the weak
convergence of the conditional distributions in \eqref{eq:occStats}
when the conditioning event is $H_{n}$ rather than
$G_{n}$. The advantage of this is that, conditional on the state of the
branching random walk at time $m (n)$, the next $n-m (n)$ generations
are gotten by running  \emph{independent} branching random walks
for time $n-m (n)$ starting from the locations of the particles in
generation $m (n)$. The argument will  hinge on showing that if
$m(n)<n$ is chosen appropriately then these
independent branching random walks will not overlap much at time $n$,
and so the total number $M_{n} (j)$ of multiplicity-$j$ sites will be,
approximately, the sum of  $Z_{m (n)}$ independent copies of
$M_{n-m(n)} (j)$.

\subsection{Overlapping}\label{ssec:overlapping}

\begin{lemma}\label{lemma:dn} Suppose that a critical branching random
walk starts at time~0 with two particles $u,v$ located at sites
$x_u, x_v \in \mathbb{Z}^d$, respectively. Let $D_n(u,v)$ be the
number of particles in generation~$n$ located at sites with
descendants of both $u$ and $v$. Then there exists $C>0$ such that
for all generations $n\geq 1$,
\begin{equation}\label{eqn:dn}
    ED_n(u,v) \leq 2P_{2n} (x_{v}-x_{u})\leq  C\left(1/\sqrt{n}\right)^d.
\end{equation}
\end{lemma}

\begin{proof}
Denote by $U^{\zeta}_{n} (x)$ the number of descendants of particle
$\zeta$ at site~$x$ in generation $n$. Since the progeny of
particles $u$ and $v$ make up mutually independent branching random
walks, the random variables $U^{u}_{n} (x)$ and $U^{v}_{n} (x)$ are
independent. But
\begin{align*}
    ED_{n} (u,v)&=E\sum_{x\in \zz{Z}^{d}} (U^{u}_{n} (x)+U^{v}_{n}
    (x))\, \indic_{\{U^{u}_{n} (x)\geq 1 \}}\, \indic_{\{U^{v}_{n} (x)\geq 1 \}}\\
    &= 2\sum_{x\in \zz{Z}^{d}} EU^{u}_n (x)\, \indic_{\{U^{v}_{n} (x)\geq 1 \}}\\
    &\leq 2\sum_{x\in \zz{Z}^{d}} P_{n} (x-x_{u})P_{n} (x-x_{v})\\
    &=2P_{2n} (x_{v}-x_{u})\\
    &\leq C\left(1/\sqrt{n}\right)^d.
\end{align*}
\end{proof}

\begin{cor}\label{corollary:overlap}
Let $Y_{n;m}$ be the number of particles in generation $n$ located at
sites with descendants of at least two distinct particles of
generation $m<n$. Then
\begin{equation}\label{eq:overlapIneq}
    E (Y_{n;m}\, |\, \mathcal{F}_{m})\leq CZ_{m}^{2}/ (n-m)^{d/2}.
\end{equation}
\end{cor}

\subsection{Convergence of means}\label{ssec:meanConvergence}

\begin{prop}\label{proposition:emn} In dimensions $d\geq 3$,
\begin{align}\label{eq:limMeans}
    \lim_n EM_n(j)&\triangleq \kappa _j  \quad \text{exists for every}
    \;\; j\geq 1, \quad \text{and}\\
\label{eq:sumLimMeans}
    \sum_{j=1}^{\infty} j\cdot\kappa _{j}&=1.
\end{align}
\end{prop}

\begin{proof}
The random variable $M_{n} (j)$ counts the number of multiplicity$-j$
sites in generation $n$. The particles at such a site will either all
be descendants of a common first-generation particle or not; hence, by
conditioning on the first generation of the branching random walk we
may decompose $M_{n+1} (j)$ as follows:
\begin{equation} \label{eqn:mn}
    M_{n+1}(j)
    = \sum_{i = 1}^{Z_1} M_n^i (j) + A_{n+1}(j)-B_{n+1}(j)
\end{equation}
where (a) the random variables $\{M^{i}_{n} (j) \}_{i\leq Z_{1}}$
are independent copies of $M_{n} (j)$; (b) the error term
$A_{n+1}(j)$ is the number of multiplicity$-j$ sites at time $n+1$
with descendants of different particles in generation $1$; and (c)
the correction $B_{n+1} (j)$ equals
$$\aligned
  \sum_{x\in\mathcal{M}_{n+1}(j+)} & \# \mbox{ particles in generation }
  1 \\
   &\quad       \mbox{ with exactly }j \mbox{ descendants at } x \mbox{
          in generation
          } (n+1),
\endaligned
$$
where $\mathcal{M}_{n+1}(j+)$ is the set of sites with $(j+1)$ or
more particles in generation $(n+1)$. Obviously, $A_{n+1} (1)=0$,
because a site with only one particle cannot have descendants of
distinct first generation particles, and so it follows that $E
M_{n+1} (1) \leq E M_{n} (1)$. This implies that $\lim_n EM_n(1)$
exists.

To see that $\lim_{n \rightarrow \infty}EM_{n} (j)$ exists for
$j\geq 2$, observe that both $A_{n+1} (j)$ and $ B_{n+1} (j)$ are
bounded by the number of $(n+1)-$th generation particles at sites
with descendants of different particles of generation $1$. Hence, by
Lemma~\ref{lemma:dn}, writing $\mathcal{Z} (1)=\mathcal{Z}_{1}$ for
the first generation of the branching process,
\begin{align}\label{eqn:bn}
 E (A_{n+1} (j) +B_{n+1} (j) )
 &\leq
2E\sum_{u,v \in \mathcal{Z}(1)} D_n(u,v)\\
\notag &\leq 2\sum_{l=2}^{\infty} Q_l {l \choose 2} Cn^{-d/2} \\
\notag &\leq C'n^{-d/2},
\end{align}
for some $C'<\infty$, because the offspring distribution has finite
second moment.  Consequently, by equation~(\ref{eqn:mn}),
\begin{equation*}
 |EM_{n+1} (j) - E M_{n} (j)| =O(n^{-d/2}).
\end{equation*}
Since the sequence $n^{-d/2}$ is summable for $d\geq 3$, the
sequence $\{EM_{n} (j) \}_{n\geq 1}$ must converge. This proves
the convergence of means \eqref{eq:limMeans}.

Clearly, for each $n\geq 1$ it is the case that $\sum_{j}jEM_{n}
(j)=EZ_{n}=1$. Hence, to prove the equation
\eqref{eq:sumLimMeans}, it suffices to show that for every
$\varepsilon >0$ there exists an integer $k=k (\varepsilon)$ such
that for all $n\geq 1$,
\begin{equation}\label{eqn:tight}
    EY_{n} (k)\leq
    \varepsilon  \quad \text{where} \quad
    Y_{n} (k)=\sum_{j=k}^\infty j\cdot M_n(j)
\end{equation}
is the number of particles in generation $n$ located at sites with
at least $(k-1)$ other particles. Since $Y_{n} (k)\leq
Z_{n}I\{Z_{n}\geq k \}$, and since $EZ_{n}=1$, it is certainly the
case that for any fixed $n\geq 1$ and $\varepsilon >0$ there
exists $k =k (n;\varepsilon)$ so that inequality \eqref{eqn:tight}
holds; the problem is to prove that $k (\varepsilon)$ can be
chosen independently of $n$.  By the same reasoning as in relation
\eqref{eqn:mn} above, for all $n,k\geq 1$,
\begin{equation}\label{eq:YnIterate}
    Y_{n+1} (k)=\sum_{u\in \mathcal{Z} (1)}Y^{u}_{n} (k) +C_{n+1} (k)
\end{equation}
where the random variables $Y^{u}_{n} (k)$ are independent copies of
$Y_{n} (k)$ and the error term $C_{n+1} (k)$ is bounded by the total
number of particles in generation $n+1$ at sites with descendants of
at least two distinct particles in $\mathcal{Z} (1)$. Since
$EZ_{1}=1$, the decomposition \eqref{eq:YnIterate} implies that
\[
    |EY_{n+1} (k)-EY_{n} (k)|\leq EC_{n+1} (k).
\]
But by the same logic as in relation \eqref{eqn:bn} above, there exists
$C'<\infty$ independent of $k$ and $n$ such that $EC_{n+1} (k) \leq C'
n^{-d/2}$ for all $n,k\geq 1$. It follows that for sufficiently large
$n (\varepsilon)$ and all $k\geq 1$,
\[
    \sum_{n=n (\varepsilon)}^{\infty} EC_{n+1} (k) <\varepsilon.
\]
Thus, if for some $k\geq 1$ and $n=n (\varepsilon)$ the inequality
\eqref{eqn:tight} holds, then $EY_{n} (k)<2\varepsilon$ for all
$n\geq n (\varepsilon)$.  This proves \eqref{eq:sumLimMeans}.
\end{proof}

\begin{remark} Since the error term $C_{n+1} (k)$ in equation
\eqref{eq:YnIterate} is nonnegative, the expectations $EY_{n} (k)$
are nondecreasing in $n$. Because the offspring distribution is
nondegenerate, for every $k\geq 1$ there exists $n\geq 1$ such
that $Y_{n} (k)\geq 1$ with positive probability, which forces
$EY_{n} (k)>0$. Therefore, there are infinitely many integers
$j\geq 1$ such that $\kappa_{j}>0$.
\end{remark}

\subsection{Conditional weak convergence: Proof of
Theorem~\ref{thm:occ3d} }\label{ssec:conditionalWeakConv} In view of
Kolmogorov's estimate \eqref{eq:survivalToGenM}, the inequality
\eqref{eqn:tight} can be rewritten as
\[
    E \left( \sum_{j\geq k} jM_{n} (j)\,\bigg|\, G_{n} \right)\leq Cn\varepsilon
\]
for some constant $C<\infty$ not depending on $n$. Since
$\Omega_{n}=\sum_{j}M_{n} (j)$, it
follows that to prove Theorem~\ref{thm:occ3d} it suffices to prove
that for any finite $k\geq 1$,
\begin{equation}\label{eq:MWeak}
    \mathcal{L}\left(\left\{\frac{M_{n} (j)}{n} \right\}_{1\leq
    j\leq k}  \, \bigg\vert \, G_{n}\right)
    \Longrightarrow \mathcal{L} (
    \left\{\kappa_{j}Y \right\}_{1\leq j\leq k})
\end{equation}
where $Y$ is exponentially distributed with mean $2/\sigma^{2}$. For
this, we will use Yaglom's theorem, the convergence of moments
\eqref{eq:limMeans}, and a crude bound on the variance of $M_{n}
(j)$:
\begin{equation}\label{eq:varBound}
    \var (M_{n} (j))\leq EZ_{n}^{2}=1+n\sigma^{2}.
\end{equation}

Fix $1\leq m<n$, and for each particle $u\in \mathcal{Z}_{m}$ let
$M^{u}_{n-m} (j)$ be the number of sites that have exactly $j$
descendants of particle $u$ in generation $n$. The random variables
$M^{u}_{n-m} (j)$ are, conditional on $\mathcal{F}_{m}$, independent
copies of $M_{n-m} (j)$. Now $M_{n} (j)$ decomposes as
\begin{equation}\label{eq:MDecomp}
    M_{n} (j) =\sum_{u\in \mathcal{Z} (m)} M^{u}_{n-m} (j)+R_{n;m}
          \triangleq M^{*}_{n;m} (j)+R_{n;m}
\end{equation}
where  the remainder $R_{n;m}$ is bounded, in absolute value, by
the number of particles in generation~$n$ located at sites with
descendants of at least two distinct particles of generation
$m<n$. By Corollary~\ref{corollary:overlap},
\begin{equation}\label{eq:remainderEst}
    E ( |R_{n;m} |\,|\, \mathcal{F}_{m}) \leq CZ_{m}^{2}/ (n-m)^{d/2}.
\end{equation}
 By Yaglom's theorem, the conditional distribution of $Z_{m}/m$ given
the event $G_{m}$ of survival to generation $m$ converges to the
exponential distribution with mean $2/\sigma^{2}$; thus, if $m=m
(n)$ is chosen so that $m /n \rightarrow 1$ and
$n-m>n^{2/(d-\varepsilon)}$ for some $\varepsilon >0$, then the
bound in \eqref{eq:remainderEst} will be of order $o_{P} (n)$.  In
view of \eqref{eq:MDecomp} and Lemmas \ref{lemma:conditioning} and
\ref{lemma:timeJiggle}, it follows that to prove \eqref{eq:MWeak}
it suffices to prove the corresponding statement in which the
random variables $M_{n} (j)$ are replaced by the approximations
$M^{*}_{n;m} (j)$ in \eqref{eq:MDecomp}, and the conditioning
events $G_{n}$ are replaced by the events $H_{n}=\{Z_{m}\geq
\varepsilon_{n}n \}$. But this follows routinely by first and
second moment estimates: if the scalars $\varepsilon_{n}$ are
chosen so that $\varepsilon_{n} \rightarrow 0$ but
$n\varepsilon_{n}/ (n-m) \rightarrow \infty$, then by relation
\eqref{eq:limMeans} and the variance bound \eqref{eq:varBound},
\[
    E\left(Z_{m}^{-1} \sum_{u\in \mathcal{Z} (m)} M^{u}_{n-m} (j)
              \, \bigg\vert \,
              \mathcal{F}_{m}\right)\indic_{H_{n}}
    \longrightarrow \kappa_{j}\indic_{H_{n}}
\]
and
\[
    \var \left(Z_{m}^{-1} \sum_{u\in \mathcal{Z} (m)} M^{u}_{n-m} (j)
              \, \bigg\vert \,
              \mathcal{F}_{m}\right)\indic_{H_{n}}
             \leq \indic_{H_{n}}(1+ (n-m)\sigma^{2})/Z_{m}
             \longrightarrow 0.
\]
Chebychev's inequality now implies that the conditional distribution
of\\ $M^{*}_{n;m} (j) /Z_{m} $ given $H_{n}$ is concentrated in a
vanishingly small neighborhood of~$\kappa_{j}$ as $n \rightarrow
\infty$.  Since the conditional distribution of $Z_{m}/n $ given
$H_{n}$ converges to the exponential distribution with mean
$2/\sigma^{2}$ by Yaglom's Theorem and Lemmas
\ref{lemma:conditioning} and \ref{lemma:timeJiggle}, the desired
result follows. \qed

\section{Typical Sites in Dimension 2: Proof of Theorem
\ref{thm:ty2d}} \label{sec:ty2d}

\subsection{Embedded Galton-Watson tree}\label{ssec:embeddedGWTree}

For simplicity we shall consider only the binary case, that is, the
special case where the offspring distribution is the double-or-nothing
distribution $Q_{0}=Q_{2}=1/2$.  The arguments can all be easily
adapted to the general case, at the expense of notational complexity.

We begin with the simple observation that the branching random walk
can be constructed by first generating a Galton-Watson tree $\tau$
according to the given offspring distribution, then
\emph{independently} attaching to the edges of this tree random steps,
distributed uniformly on the set $\mathcal{N}$ of nearest neighbors of
the origin. The vertices of $\tau $ at height $n$ represent the
particles of generation $n$; the location in $\zz{Z}^{2}$ of a
particle $\alpha$ of the $n$th generation is obtained by summing the
random steps on the edges of the path in $\tau $ leading from the root
to $\alpha$.  Henceforth we will distinguish between the underlying
Galton-Watson tree $\tau $ and the \emph{marked} tree $\tau^{*}$
obtained by attaching step variables to the edges of $\tau$. Observe
that the conditional distribution of the marks of $\tau^{*}$ given the
tree $\tau$ is the product uniform measure on
$\mathcal{N}^{\mathcal{E} (\tau)}$, where $\mathcal{E} (\tau)$ denotes
the set of edges of $\tau$.

A \emph{typical particle} of the $n$th generation in a branching
random walk conditioned to survive to the $n$th generation can be
obtained by first choosing a tree $\tau$ randomly according to the
conditional distribution $F_{n}$ of the Galton-Watson tree given the
event of survival to generation $n$, then randomly selecting one of
the $Z_{n}\geq 1$ vertices at height $n$. For this random choice we
assume that the underlying probability space supports a
uniform-$[0,1]$ random variable $\gamma$ independent of all other
random variables used in the construction of the branching random
walk. Since this procedure does not use information about the step
variables attached to the edges of the tree, it follows directly
that the trajectory of the typical particle, conditional on the
underlying Galton-Watson tree, is a simple random walk started at
the origin.

\subsection{Reduction to the Size-Biased
Case}\label{ssec:reductionSB}

The strategy of the proof of Theorem~\ref{thm:ty2d} will be based on a
change of measure. Denote by $P_{H}=P_{H}^{n}$ the probability measure
that is absolutely continuous relative to $P$ with Radon-Nikodym
derivative
\begin{equation}\label{eq:rnDerivative}
    \frac{dP_{H}}{dP} = Z_{n} .
\end{equation}
The measure $P_{H}^{n}$ so defined is a probability measure, because
$EZ_{n}=1$. Call it the \emph{size-biased} measure. In the arguments
below the value of $n$ will be fixed, so we will generally omit the
dependence of the measure on $n$ and write $P_{H}=P_{H}^{n}$.
Because the Radon-Nikodym derivative depends only on the underlying
Galton-Watson tree $\tau$, which under $P$ is independent of the
marks, it follows that the conditional distribution under $P_{H}$ of
the marks given the tree $\tau$ is the same as under $P$. Thus, to
construct a version of the marked tree $\tau^{*}$ under $P_{H}$, one
may first build a size-biased version of the underlying
Galton-Watson tree, then attach edge marks independently according
to the (product) uniform distribution on $\mathcal{N}$. Henceforth
we will call such a marked tree a \emph{size-biased marked tree} or
a \emph{size-biased branching random walk}.

Observe that $P_{H}$ is also absolutely continuous relative to the
\emph{conditional} distribution $P^{*}_{n}$ of $P$ given the event
$G_n$ of survival to generation $n$; the Radon-Nikodym derivative
is
\begin{equation}\label{eq:condRN}
        \frac{dP_{H}}{dP^{*}_{n}} = Z_{n} \pi_{n}
\end{equation}
where $\pi_{n}=P (G_{n})\sim (2/n\sigma^{2})$. By Yaglom's theorem,
under $P^{*}_{n}=P (\cdot \,|\,G_n)$ the distribution of $dP_{H}^{n}/dP^{*}_{n}$
converges in law to the unit exponential distribution. This  implies
the following.

\begin{lemma}\label{lemma:reduction}
To prove Theorem~\ref{thm:ty2d} it suffices to prove the analogous
statements for the measure $P_{H}$, that is, to prove that (i) for
each $\varepsilon >0$ there exists $K<\infty$ such that
\begin{equation}\label{eq:obj(i)}
    P_{H} \{ T_{n}\geq K \log n\} <\varepsilon ;
\end{equation}
and (ii) for all sufficiently small $\varepsilon >0$ there exists
$\delta >0$ such that for all large $n$,
\begin{equation}\label{eq:obj(ii)}
    P_{H}\{ T_{n}\geq \varepsilon  \log n\} \geq \delta .
\end{equation}
\end{lemma}

\begin{proof}
This is a direct consequence of the fact that the Radon-Nikodym
derivatives $dP_{H}/dP^*$ converge in law under $P^*$ as $n
\rightarrow \infty$, because this implies that the Radon-Nikodym
derivatives $dP^*/dP_{H}$ converge in law under~$P_{H}$.
\end{proof}

\subsection{Structure of the size-biased
process}\label{ssec:size-biasedStructure} The size-biased measure
$P_{H}$ on marked trees is especially well-suited to studying typical
points, and has been used by a number of authors (see
\cite{lyonsPemantlePeres} and the references therein) for similar
purposes. Consider first the distribution of the \emph{unmarked}
genealogical tree $\tau$ under $P_{H}$. According to
\cite{lyonsPemantlePeres}, a version of this random tree can be
obtained by running a certain \emph{Galton-Watson process with
immigration}. In the case of the double-or-nothing offspring
distribution, the nature of this process is especially simple:

\medskip \noindent \textbf{Recipe SB:} Each generation $j$ has
a single distinguished particle $v_{j}$ which gives rise to two
particles in generation $j+1$, one the distinguished particle
$v_{j+1}$, the other an undistinguished particle. All
undistinguished particles reproduce according to the
double-or-nothing law. For each $n$, the distinguished particle
$v_n$ is uniformly distributed on the particles in generation~$n$.

\qed

Thus, a version of the size-biased branching random walk, together
with a randomly chosen point $v_{n}$ of the $n$th generation, can be
built by attaching independent  step random variables to the edges of
the random tree built according to Recipe SB. Equivalently, this
process can be  constructed
using three independent sequences of auxiliary
random variables:
\begin{itemize}
   \item [($T_a)$] $\{S_{n} \}_{n\geq 0}$ is a simple random walk  in
$\zz{Z}^{2}$ with initial point $S_{0}=0$;
  \item [($T_b)$] $\{\xi_{i}
\}_{i\geq 0}$ are independent and  uniformly distributed on
$\mathcal{N}$;
   \item [($T_c)$] $\{U^{i}_{n} (x) \}_{i\geq 0}$ are independent
copies of the branching random walk $\{U_{n} (x) \}$ run according
to the law $P$; and
   \item [($T_d)$] $B_0\sim$ Bernoulli$(1/(2d+1))$.
\end{itemize}
(We emphasize that the auxiliary branching random walks $\{U^{i}_{n}
(x) \}_{i\geq 1}$ are run according to the original probability
measure $P$, not the size-biased measure $P_{H}$.) The size-biased
branching random walk is obtained by letting the ``typical'' particle
follow the trajectory $S_{j}$, then attaching an additional particle
to each point $(j,S_{j})$ visited by the typical particle, letting it
make a step to $S_{j}+\xi_{j}$, and then attaching the $j$th copy of
the branching random walk $U^{j}$ to this particle.

\begin{cor}\label{corollary:representationSB}
The distribution of $T_{n}$ under the size-biased measure
$P_{H}$ is the same as the distribution under $P$ of the random
variable
\begin{equation}\label{eq:representationSB}
    T^{*}_{n}= 1+ B_0 +  \sum_{j=0}^{n-2} U^{j}_{n-j-1} (S_{n} - S_{j}
    -\xi_{j}) .
\end{equation}
\end{cor}
\qed

The Bernoulli random variable $B_0$ accounts for the possibility
that the sibling of the typical particle jumps to the same site as
the typical particle.

Reversing the random walk will not affect the distribution of the
random variable $T_{n}$, since the random walk is independent of all
other component variables of the representation
\eqref{eq:representationSB}, nor will reversing the indices of the
auxiliary branching random walks $U^{j}$. Thus, the following random
variable has the same distribution as that given by
\eqref{eq:representationSB}:
\begin{equation}\label{eq:theRepresentation}
        T^{**}_{n}= 1+ B_0 + \sum_{j=2}^{n}U^{j-1}_{j-1} (S_{j}+\xi_{j-1}).
\end{equation}

\subsection{Variances of the occupation random variables} Next we
focus on the distribution of the random variable $T^{**}_{n}$ defined by
\eqref{eq:theRepresentation}.  To obtain concentration results for this
distribution, we will need bounds on the second moments of the random
variables $U_{n} (x)$; for this, we use an exact formula for the
second moment of $U_{n} (x)$, valid in all dimensions:

\begin{prop}\label{proposition:un2}

\begin{equation}\label{eqn:un2nd}
   EU_n(x)^2=P_n(x)+\sigma^2\sum_{i=0}^{n-1}\sum_z P_i(z)P_{n-i}^2(x-z),
\end{equation}
\end{prop}

\begin{proof}
This is a special case of equation (81) in \cite{lalley07}, which
gives the $m$th moment for all integers $m\geq 1$. In the case $m=2$,
a simple proof can be given by conditioning on the first generation of
the branching random walk. Set $f_{n} (x) = EU_n(x)^2$ and $g_{n}
(x)=P_{n} (x)^{2}$; then conditioning on generation $1$ gives
\[
    f_{n} (x)=\zz{P}f_{n-1} (x)+\sigma^{2}g_{n} (x).
\]
Since the operator $\zz{P}$ is linear, this relation can be iterated
$n-1$ times, yielding
\[
    f_{n} (x)=\zz{P}^{n-1}f_{1} (x) +\sigma^{2}\sum_{i=0}^{n-2} \zz{P}^{i}g_{n-i} (x).
\]
This is equivalent to the  identity \eqref{eqn:un2nd}.
\end{proof}

\begin{remark}\label{rmk:2nd_moment}
If the offspring distribution has an exponentially decaying tail,
then one can deduce from (\ref{eqn:2d_mgf}) that $\sum_x
EU_n(x)^2\leq C\log n/\theta $. However, formula (\ref{eqn:un2nd})
implies that $\sum_x EU_n(x)^2 $  grows at rate  $\log n$, so
(\ref{eqn:2d_mgf}) cannot hold for large $\theta.$
\end{remark}

\subsection{Mean and variance estimates for
$T^{**}_{n}$}\label{ssec:concentration}The sum in the representation
\eqref{eq:theRepresentation} can be decomposed as $\Gamma_{n}+\Delta_{n} $,
where
\begin{align}\label{eq:GammaN}
    \Gamma_{n}:&= \sum_{i=2}^{n}P_{i} (S_{i}) \quad
    \text{and} \\
\label{eq:DeltaN}
    \Delta_{n} :&= \sum_{i=2}^{n} X_{i-1}
    \qquad \text{with} \quad
    X_{i-1}:=U^{i-1}_{i-1} (S_{i}+\xi_{i-1}) -P_{i} (S_{i}).
\end{align}

\begin{lemma}\label{lemma:GammaN}
Let $S_{n}$ be simple random walk in $\zz{Z}^{2}$, and let
$\Gamma_{n}$ be defined by \eqref{eq:GammaN}. Then
\begin{align}\label{eq:limEGamma}
    \lim_{n \rightarrow \infty} &\frac{E\Gamma_{n}}{\log n}=
        \frac{A}{2}
        \quad \text{and} \\
\label{eq:limVGamma}
    \lim_{n \rightarrow \infty}
    &\var \left(\frac{\Gamma_{n}}{\log n}\right)=0.
\end{align}
\end{lemma}

Recall that $A=5/(4\pi)$ is the  constant such that $P_n(0)\sim
A/n$.
\begin{proof}
By the symmetry of the simple random walk, $EP_{i} (S_{i})=P_{2i}
(0)\sim A/ (2i)$, and so the first convergence  \eqref{eq:limEGamma}
follows routinely.  To estimate the variance, first observe that
\begin{align}
\notag
    E\Gamma_{n}^{2}&= 2\sum_{i< j}EP_i(S_i)P_j(S_j)
                      +\sum_{i=2}^n E P_i(S_i)^2 \\
\label{eq:covSum}
                      &=2\sum_{i<j}EP_i(S_i)P_j(S_j) +O (1).
\end{align}
The second equation follows from the local central limit theorem
in $d=2$, which guarantees that $P_{i} (z)\leq C/i$ for some
constant $C<\infty$ independent of $i$ and $z$. Next, observe that
for $i<j$, by the symmetry of the random walk and the fact that
$P_{i} (z)$ is maximal at $z=0$ (Lemma~\ref{lemma:um0})
\begin{align}\label{eq:covTerm}
    EP_i(S_i)P_j(S_j)&=E(E(P_i(S_i)P_j(S_j)|S_i))\\
\notag &=EP_i(S_i)\sum_{x\in \mathbb{Z}^2} P_j(S_i+x)P_{j-i}(x)\\
\notag &=EP_i(S_i)P_{2j-i}(S_i)\\
\notag &\leq EP_i(S_i) P_{2j-i} (0)\\
\notag &= P_{2i} (0)P_{2j-i} (0).
\end{align}
Substituting this bound in \eqref{eq:covSum} and applying the local
central limit theorem (in the form $P_{n} (0)\sim A/n$) yields
\begin{align*}
        \sum_{i<j}EP_i(S_i)P_j(S_j)
             &\leq \sum_{j=2}^{n}\sum_{i<j}
              P_{2i} (0)P_{2j-i} (0) \\
            &\leq 2\sum_{j=2}^{n}\sum_{i<j}
             A^{2}/ (2i (2j-i)) + \text{error}\\
            &\sim \frac{A^{2}}{4}\log^{2}n +\text{error},
\end{align*}
where the error is of smaller order of magnitude. Together with
\eqref{eq:covSum} and \eqref{eq:limEGamma}, this shows that
\[
    \var \left({\Gamma_{n}}\right)
    =E\Gamma_{n}^{2} - (E\Gamma_{n})^{2}
    =o (\log n)^{2}.
\]
\end{proof}

\begin{lemma}\label{lemma:DeltaN}
Let $S_{n}$,  $U^{i}_{n} (x)$, and $\xi_{i}$ be
independent sequences of random variables satisfying the
hypotheses $(T_a)-(T_c)$ of section~\S\ref{ssec:size-biasedStructure}.
If $\Delta_{n}$ and $X_{i}$ are
defined as in equation~\eqref{eq:DeltaN}, then
\begin{equation}\label{eq:noCorrelation}
    EX_{i}=0 \quad \text{and} \quad EX_{i}X_{j}=0 \quad \text{for
    all} \; i\not =j.
\end{equation}
Consequently,
\begin{equation}\label{eq:varDeltaN}
 E \Delta_{n}=0 \quad \text{and} \quad
    \lim_{n \rightarrow \infty} \var\left(\frac{\Delta_n}{\log
    n}\right)
     =\frac{A^{2}}{8}.
\end{equation}
\end{lemma}

\begin{proof}
To show that $E\Delta_{n}=0$ it suffices to show that $EX_{i}=0$. This
follows from the fundamental relation \eqref{eq:fundamental} by
conditioning on $S_{i+1}$ and $\xi_{i}$:
\begin{align*}
    EX_{i}&=EE (U^{i}_{i} (S_{i+1}+\xi_{i})\,|\, S_{i+1},
           \,\xi_{i}) -EP_{i+1} (S_{i+1})\\
           &=EP_{i}(S_{i+1}+\xi_{i})-EP_{i+1} (S_{i+1})=0.
\end{align*}
Now consider the covariances $EX_{i}X_{j}$. To compute these
expectations for $i<~j$, condition on the random variables
$S_{i+1}, S_{j+1},\{ U^{i}_{i} (x)\}_{x\in \zz{Z}^{2}}$, and
$\xi_{i}$ (but not $\xi_{j}$), and use the fundamental
identity~\eqref{eq:fundamental}: This implies that $EU_{j}
(x+\xi_{j})=P_{j+1} (x)$ for each $x\in \zz{Z}^{2}$, and so
\begin{align*}
    EX_{i}X_{j}&=EE (X_{i}X_{j}\,|\,\cdot)\\
              &=EX_{i}E (U^{j}_{j} (S_{j+1}+\xi_{j})-P_{j+1} (S_{j+1})\,|\, \cdot)\\
              &=EX_{i}\cdot 0 = 0.
\end{align*}
It follows that the variance of the sum $\Delta_{n}$ is the sum of
the variances of the increments $X_{i}$, and~so
\begin{align*}
    \var (\Delta_{n})&=\sum_{i=2}^{n}E X_{i-1}^{2}
         =\sum_{i=2}^{n} EX_{i-1}^{2}\\
         &=\sum_{i=2}^{n} (EU^{i-1}_{i-1} (S_{i}+\xi_{i-1})^{2}-EP_{i} (S_{i})^{2})\\
         &=\sum_{i=2}^{n} EU^{i-1}_{i-1}
         (S_{i}+\xi_{i-1})^{2} +O (1).
\end{align*}
Now by the second moment
formula~\eqref{eqn:un2nd},
\begin{align*}
    &EU^{i-1}_{i-1} (S_{i}+\xi_{i-1})^{2}\\
=&E\left(P_{i-1}(S_{i}+\xi_{i-1})+\sum_{j=1}^{i-1}\sum_z
P_j(z)^2P_{i-j-1}(S_{i}+\xi_{i-1}-z)\right)\\
=&EP_{i}(S_{i}) +\sum_{j=1}^{i-1}\sum_z P_j(z)^2\cdot
EP_{i-j}(S_i-z)\\
=&P_{2i}(0) +\sum_{j=1}^{i-1}\sum_z P_j(z)^2\cdot P_{2i-j}(z).
\end{align*}
The first term is of order $O (1/i)$. To estimate the second, observe
that by the local central limit theorem, for large $j$,
\[
    P_{j} (z)^{2}\sim\frac{A}{2j}P_{[j/2]} (z)
\]
where $[\cdot]$ denotes integer part and the relation holds
uniformly for $|z|\leq C\sqrt{j}$. Consequently, for large $i$,
\begin{align*}
    \sum_{j=1}^{i-1}\sum_z P_j(z)^2\cdot P_{2i-j}(z)
    &\sim \sum_{j=1}^{i-1} \frac{A}{2j}\sum_z
          P_{[j/2]} (z) P_{2i-j} (z)\\
    &= \sum_{j=1}^{i-1} \frac{A}{2j} P_{2i-j+[j/2]} (0)\\
    &\sim \sum_{j=1}^{i-1} \frac{A}{2j} \frac{A}{2i-j/2} \\
    &\sim \frac{A^{2}\log i}{4i}.
\end{align*}
Summing from $i=1$ to $n$ a shows that $
\var(\Delta_{n})\sim ( A^{2}/8)\log^{2}n$. This
proves~\eqref{eq:varDeltaN}.
\end{proof}

\subsection{Proof of Theorem~\ref{thm:ty2d}: Binary fission
case}\label{ssec:Th5DN} By Lemma~\ref{lemma:reduction}, it
suffices to prove assertions
\eqref{eq:obj(i)}--\eqref{eq:obj(ii)}. By
Corollary~\ref{corollary:representationSB}, the distribution of
$T_{n}$ under the size-biased measure $P_{H}$ is identical to the
distribution of the random variable
$T^{**}_{n}:=1+B_0+\tilde{T}_{n}$ under $P$, where $T^{**}_{n}$ is
defined by \eqref{eq:theRepresentation}. Finally, by Lemmas
\ref{lemma:GammaN} and \ref{lemma:DeltaN} (note that $E\Delta_n
\Gamma_n=0$),
\[
    E\tilde{T}_{n} \sim \frac{A}{2}\log n \quad \text{and} \quad
    \var (\tilde{T}_{n})\sim \frac{A^2}{8} \log^{2}n.
\]
The first of these implies, by the Markov inequality, that
$T^{**}_{n}=O_{P} (\log n)$. This proves the first assertion (i) of
Lemma~\ref{lemma:reduction}. The second assertion (ii)
is a consequence of the following elementary lemma (see, e.g.,
\cite{lawler07}, Lemma 12.6.1).

\begin{lemma}\label{lemma:meanVar}
If $X$ is a nonnegative random variable with positive, finite second
moment, then for any $\alpha \in [0,1]$,
\begin{equation}\label{eq:meanVar}
    P\{X\geq \alpha EX \}\geq (1-\alpha)^{2} (EX)^{2}/EX^{2}.
\end{equation}
\end{lemma}

\qed

\section{Clustering in Dimension 2: Proof of
Theorem~\ref{thm:ball2d}}\label{sec:clustering}

\subsection{Occupied sites in the ball $B
(S_{n};\ell_n)$}\label{ssec:occSitesBall} We consider only the
case of binary fission. The proof of Theorem~\ref{thm:ball2d} in
this case, like that of Theorem~\ref{thm:ty2d}, is based on the
change of measure strategy outlined in section~\ref{ssec:reductionSB}.
In particular, we shall prove the corresponding assertions to
statements (A)-- (B) of Theorem~\ref{thm:ball2d} for the
\emph{size-biased} process of
section~\ref{ssec:size-biasedStructure}. Thus, assume throughout this
section that the random variables $S_{j}$, $U^{j}_{k}$, and $\xi_{j}$
are as in ($T_{a}$), ($T_{b}$), ($T_{c}$) of
section~\ref{ssec:size-biasedStructure}. Recall that the size-biased
branching random walk is obtained by letting the ``typical'' particle
follow the trajectory $S_{j}$, then attaching an additional particle
to each point $(j,S_{j})$ visited by the typical particle, letting it
make a step to $S_{j}+\xi_{j}$, and then attaching the $j$th copy of
the branching random walk $U^{j}$ to this particle. To prove
Theorem~\ref{thm:ball2d} it suffices to prove the following
proposition.

\begin{prop}\label{proposition:ClusterSB}
 Let $\{\ell_n\}$ be
any sequence of real numbers such that $\lim_n\ell_n=\infty$ and
$\lim_n \log \ell_n/\!\log n\!=\!0$. Let $B (S_{n};\ell_n)$ be the ball of radius
$\ell_n$ centered at $S_{n}$. Then for the size-biased branching
random walk,
\begin{enumerate}
\item [(A)] the number of \emph{unoccupied} sites in $B
(S_{n};\ell_n)$ is $o_{P} (\ell_n^2)$, and
\item [(B)] the number of particles in $B (S_{n};\ell_n)$ is of order
$O_p(\log n\cdot \ell_n^2)$.
\end{enumerate}
\end{prop}

The construction of section~\ref{ssec:size-biasedStructure} shows
(cf. formulas~\eqref{eq:representationSB} and
\eqref{eq:theRepresentation}) that the number of particles at location
$S_{n}+x$ in the $n$th generation of the size-biased branching random
walk is distributed as
\begin{equation}\label{eq:unStar}
    U^{**}_{n} (S_{n}+x):=\delta_{0} (x) + B_0\cdot\indic_{\{|x|\leq 1\}} + \sum_{j=1}^{n-1}
    U^{j}_{j} ( S_{j+1} +x    +\xi_{j}).
\end{equation}

\subsection{Vacant Sites: Proof of Theorem~\ref{thm:ball2d}
(A)}\label{ssec:vacancies} The representation
\eqref{eq:unStar} implies that the
probability that the site $x+S_{n}$ is unoccupied, that is, that
$U^{**}_{n} (x+S_{n})=0$, is equal to the
probability that none of the branching random walks $U^{i}_{i}$
succeeds in placing a particle at location $x$ at time $n$.  Since the
attached branching random walks are independent of the random walk
trajectory $\{S_{i} \}_{i\leq n}$ and the  displacement random variables $\xi_{i}$,
this probability is
\begin{equation}\label{eq:vacancyProb}
    P \{\text{site} \,(S_{n}+x) \, \text{vacant} \}
    =
    \prod_{i} \left(1-u_{i} (x+S_{i+1}+\xi_{i})\right)
\end{equation}
where $u_{n}$ is the \emph{hitting probability function}
\begin{equation}\label{eq:un}
    u_{n} (x):=P\{U_{n} (x)\geq 1 \}.
\end{equation}

\begin{prop}\label{proposition:ulowerbound}
There exists $C>0$ such that for all $n\geq 1$ and all sites $x\in
\zz{Z}^{2}$,
\begin{equation}\label{eq:ulowerbound}
    u_{n} (x)\geq \frac{P_{n} (x)}{C+A \log n}.
\end{equation}
\end{prop}

\begin{proof}
By the fundamental identity, $EU_{n} (x)=P_{n} (x)$.
By the second moment formula \eqref{eqn:un2nd} of
Proposition~\ref{proposition:un2},
\begin{align}\label{eq:2ndMomentBound}
    EU_{n} (x)^{2}
    &=P_n(x)+\sum_{i=0}^{n-1}\sum_z     P_i(z)P_{n-i}^2(x-z) \\
\notag &  \leq P_n(x)+\sum_{i=0}^{n-1}\sum_z P_{i} (z)P_{n-i} (x-z) P_{n-i} (0)\\
\notag  &= P_{n} (x)+P_{n} (x)\sum_{i=0}^{n-1}P_{n-i} (0) \\
\notag  &\leq P_{n} (x) (C+A\log n).
\end{align}
Here we have used the fact (Lemma~\ref{lemma:um0}) that $P_{n-i} (x)$
is maximal at the origin $x=0$, together with a strong form of the
local central limit theorem (specifically, the fact that the error in
the local limit approximation is of order $O (n^{-2})$, which is
summable).  The result \eqref{eq:ulowerbound} now follows immediately
from the Cauchy-Schwartz inequality $P\{X>0 \}\geq (EX)^{2}/EX^{2}$,
valid for any nonnegative random variable $X$.
\end{proof}

The lower bound \eqref{eq:ulowerbound} leads easily to a useful
\emph{upper} bound for the probability that site
$x$ is vacant. Partition the indices $i\leq n$ into two sets,
the \emph{good} and the \emph{bad} indices, as follows: Fix a large
constant $\kappa <\infty$, and say that index $i$ is \emph{good} if
$|S_{i+1}+\xi_{i}|\leq \kappa \sqrt{i}$; say that $i$ is
\emph{bad} otherwise. By the local central limit theorem, there is a
constant $C'>0$ not depending on $\kappa $ such that for every good
index $i\geq |x|^{2}$,
\begin{equation}\label{eq:helpLLT}
    P_{i} (x+S_{i+1}+\xi_{i}) \geq C' e^{-2\kappa^{2}}/i.
\end{equation}
Thus, relations \eqref{eq:ulowerbound}--\eqref{eq:vacancyProb} and the
concavity of the logarithm function imply that for a suitable constant
$C''>0$ not depending on $\kappa $,
\begin{equation}\label{eq:vacancyProbBound}
       P \{\text{site} \,(S_{n}+x) \, \text{vacant} \}
    \leq
    \exp \left\{-C'' e^{-2\kappa^{2}}\!\!\!\!\sum_{i \,\text{\rm good}, \,
     |x|^{2}\leq i\leq n } \frac{1}{i\log i} \right\}.
\end{equation}

\begin{lemma}\label{lemma:good-bad}
Let $\{\ell_n\}$ be any sequence of real numbers such that
$\lim_n\ell_n=~\infty$ and $\lim_n \log \ell_n/\log n$ $=0$. Then
for every $b>0$ and  every $\varepsilon >0$ there exists $\kappa$
sufficiently large that
\begin{equation}\label{eq:good-bad}
  \limsup_n  P\left\{ \sum_{i \,\text{ \rm good}, \,
    \ell_n^{2}\leq i\leq n } \frac{e^{-2\kappa^{2}}}{i\log i}
    \leq b \right\}<\varepsilon.
\end{equation}
\end{lemma}

\begin{proof}
The hypotheses regarding the growth of $\ell_n$ ensure that
\[
    L_{n}:=\sum_{i=\ell_{n}^{2}}^{n} 1/ (i \log i) \longrightarrow \infty .
\]
Hence, it suffices to show that for some $0<\varrho<1$, if $\kappa$ is
sufficiently large then
\begin{equation}\label{eq:good-bad-Obj}
    P\left\{ \sum_{i \,\text{bad}, \,
    \ell_n^{2} \leq i\leq n} \frac{1}{i\log i}  \geq \varrho
    L_{n}\right\}   <\varepsilon
\end{equation}
for all large $n$.  Recall that an index $i$ is \emph{bad} if
$|S_{i+1}+\xi_{i}|>\kappa \sqrt{i}$. Chebyshev's inequality implies
that for any $\varepsilon >0$, if $\kappa$ is sufficiently large
then $P\{ |S_{i+1}+\xi_{i}|>\kappa \sqrt{i}\}<\varepsilon^{3}$;
hence, for large $n$,
\[
    E \sum_{\ell_n^{2}\leq i\leq n } \frac{\indic_{\{|S_{i+1}+\xi_{i}|>\kappa
              \sqrt{i}\}}}{i\log i}
    \leq \varepsilon^{3}L_{n}.
\]
It now follows by the Markov inequality that
\begin{equation}\label{eq:badRW}
    P\left\{ \sum_{\ell_n^{2}\leq i\leq n } \frac{\indic_{\{|S_{i+1}+\xi_{i}|>\kappa
              \sqrt{i}\}}}{i\log i}  \geq \varepsilon
    L_{n}\right\} \leq \varepsilon^{2}.
\end{equation}
The relations \eqref{eq:badRW} clearly implies
\eqref{eq:good-bad-Obj}, and therefore prove \eqref{eq:good-bad}.
\end{proof}

\begin{proof}
[Proof of Proposition~\ref{proposition:ClusterSB}(A)] For any
$\varepsilon >0$, inequality \eqref{eq:vacancyProbBound} and
Lemma~\ref{lemma:good-bad} imply that for all large $n$, for any
displacement $x$ of magnitude~$\leq \ell_n$, the probability that
site $(x+S_{n})$ is vacant  is less than $2\varepsilon$. Therefore,
the expected number of vacant sites in the ball $B (S_{n};\ell_n)$
given the event $G_{n}$ is, for large $n$, no larger than $4\pi
\varepsilon \ell_{n}^{2}$. The assertion (A) of
Theorem~\ref{thm:ball2d} follows directly, by the Markov inequality.
\end{proof}

\subsection{Proof of Proposition~\ref{proposition:ClusterSB} (B)}\label{ssec:ball2d}
The second assertion (B) of Proposition~\ref{proposition:ClusterSB} can be proved in
virtually the same manner as Theorem~\ref{thm:ty2d}. Following is a
brief sketch. Set
\begin{equation}\label{eq:noInBall}
\aligned
    W_{n}:=\# &\text{ particles of generation} \;n
    \;\;\text{within
    distance}\; \ell_{n} \;\text{of}\;\; S_{n}\\
    &\text{in the size-biased BRW}.
\endaligned
\end{equation}
By representation~\eqref{eq:unStar},
\begin{equation}\label{eq:representationBall}
       W_{n}
   = 2+  \sum_{i=1}^{n-1} \sum_{|x|\leq \ell_n}
      U^i_i(x+S_{i+1}+\xi_i),
\end{equation}
where   $U^{i}_{j} (x)$, $S_{n}$, and $\xi_i$ satisfy
conditions ($T_{a}$)--($T_{c}$) of section
\S~\ref{ssec:size-biasedStructure}. The distribution
of the sum on the right side is analyzed by decomposing it as
$\Gamma_{n}+\Delta_{n}$, where now
\begin{align}\label{eq:newGammaDelta}
    \Gamma_{n}&:=\sum_{i=2}^n \sum_{|x|\leq \ell_n}
         P_i(x+S_i)\quad \text{and}\\
\notag
     \Delta_{n}&:= \sum_{i=2}^n  \left(\sum_{|x|\leq \ell_n}
           (U^{i-1}_{i-1}(x+S_{i}+\xi_{i-1})-P_i(x+S_i))\right).
\end{align}
By calculations similar to those used in proving
Lemmas~\ref{lemma:GammaN}, one shows that
\begin{align}\label{eq:meanVarEstsBall}
        \lim_{n \rightarrow \infty} &E\Gamma_{n}/ (\pi \ell_{n}^{2}\log
        n)=A/2;\\
\notag  \lim_{n \rightarrow \infty} &\var (\Gamma_{n})/ (\pi \ell_{n}^{2}\log n)=0;\\
\notag  \lim_{n \rightarrow \infty} &\var (\Delta _{n})/ (\pi
        \ell_{n}^{2}\log n)\leq A^{2}/8; \quad \text{and}\\
\notag  E\Delta_{n}&=0 \qquad  \text{for all} \; n\geq 1.
\end{align}
Given these estimates, one now obtains the desired conclusion, that
$W_{n}$ is of order $O_{P} (\ell_n^{2}\log n)$, by the
same simple argument as in section~\ref{ssec:Th5DN}.
\qed

\newpage

\section{Occupied Sites in Dimension 2}\label{sec:occ2d}

\subsection{Hitting probability function}\label{ssec:hittingProbs} For
simplicity we consider in this section only the binary fission case;
the case of a general offspring distribution with mean $1$ and finite
variance can be handled similarly. The proof of
Theorem~\ref{thm:occ2d} will be based on careful analysis of the
hitting probability function $u_{n} (x)$ defined by equation
\eqref{eq:un} above.  The connection with the total number
$\Omega_{n}$ of occupied sites at time $n$ is obvious:
$E\Omega_{n}=\sum_{x}u_{n} (x)$. Thus, our goal will be to bound the
function $u_{n}$ from \emph{above}. (A good \emph{lower} bound has
already been obtained in Proposition~\ref{proposition:ulowerbound}.)
Our main result is the following proposition.

\begin{prop}\label{proposition:uupperBound}
There exist constants  $C_1,C_2<\infty$ such that for all $n\geq
2$ and all sites $x\in \zz{Z}^{2}$,
\begin{equation}\label{eq:upperBoundu}
     u_{n} (x)\leq \frac{C_1}{n \log  n}\exp\left(-C_2
     \frac{|x|^2}{n}\right),
\end{equation}
and hence for some $C>0$ we have that
\begin{equation}\label{eq:upperBoundu}
    E\Omega_n=\sum_x u_{n} (x)\leq \frac{C}{ \log  n}.
\end{equation}
\end{prop}

Theorem~\ref{thm:occ2d} follows as a direct consequence of
\eqref{eq:upperBoundu} and Kolmogorov's estimate
\eqref{eq:survivalToGenM}.

To obtain upper bounds on the function $u_{n} (x)$, we will exploit
the fact that it satisfies a parabolic nonlinear partial difference
equation. Recall that $\zz{P}$ is the Markov operator for the simple
random walk, that is, for any bounded function $w:\zz{Z}^{2}
\rightarrow \zz{R}$,
\[
    \zz{P}w (x)= \frac{1}{5}\sum_{z-x\in \mathcal{N}} w (z).
\]

\begin{lemma}\label{lemma:hittingProbEqn}
Assume that the offspring distribution is double-or-nothing. Then
for each $n\geq 0$ and each $x\in \zz{Z}^{d}$,
\begin{equation}\label{eq:kpp}
    u_{n+1} (x)=\zz{P}u_{n} (x)-\frac{1}{2} (\zz{P}u_{n} (x))^{2}.
\end{equation}
\end{lemma}

\begin{proof}
The event $\{U_{n+1} (x)>0 \}$ can only occur if the first
generation is nonempty, and hence consists of two particles with
locations in $\mathcal{N}$. This happens with probability $1/2$. One
or both of these particles must then engender a descendant branching
random walk that places a particle at site~$x$ in its $n$th
generation. Since the two descendant branching random walks are
independent, with starting points randomly chosen from
$\mathcal{N}$, this happens with probability $2p (1-p) +p^{2}$,
where $p=\zz{P}u_{n} (x)$.
\end{proof}

To extract information from the nonlinear difference equation
\eqref{eq:kpp} we will use the following standard comparison
principle. (Compare, for example, Proposition 2.1 of
\cite{aronson-weinberger}.)

\begin{lemma}\label{lemma:comparison}
Let $u_{n} (x)$ and $v_{n} (x)$ be functions taking values between $0$
and $1$ that satisfy the following conditions:
\begin{align}\label{eq:u}
    u_{n+1} (x)&=\zz{P}u_{n} (x)-\frac{1}{2} (\zz{P}u_{n} (x))^{2} \quad \text{and}\\
\label{eq:v}
    v_{n+1} (x)&\geq \zz{P}v_{n} (x)-\frac{1}{2} (\zz{P}v_{n} (x))^{2}.
\end{align}
If  $v_{0} (x)\geq u_{0} (x)$ for all $x$, then
\begin{equation}\label{eq:supersolution}
    v_{n} (x)\geq u_{n} (x) \quad \text{for all} \; n\geq 0
    \;\; \text{and} \;x\in \zz{Z}^{2}
\end{equation}
\end{lemma}

\begin{proof}
Define $\Delta_{n} (x)=v_{n} (x)-u_{n} (x)$; then by the
hypotheses \eqref{eq:u}--\eqref{eq:v},
\begin{equation}\label{eq:delta}
    \Delta_{n+1} (x)\geq \zz{P}\Delta_{n} (x)- \frac{1}{2}
             (\zz{P}u_{n} (x) +\zz{P}v_{n} (x)) \zz{P}\Delta_{n} (x).
\end{equation}
Since $u_{n}$ and $v_{n}$ take values between $0$ and $1$, so does the
average $ (\zz{P}u_{n} +\zz{P}v_{n})/2$. Therefore, \eqref{eq:delta}
and the induction hypothesis imply
\[
    \Delta_{n+1} (x)\geq \zz{P}\Delta_{n} (x) \left(1- \frac{1}{2}
             \left(\zz{P}u_{n} (x) +\zz{P}v_{n} (x)\right)\right) \geq 0.
\]
\end{proof}

The trick is to find a function $v_{n}$ that satisfies inequality
\eqref{eq:v} and dominates $u_{0}$. To this end, fix $\kappa >0$
and define
\begin{equation}\label{eq:v_n}
 v_n(x)=\frac{\kappa}{n\log n}\exp\left(-\frac{\beta_n
  |x|^2}{2n}\right),
\end{equation}
where
\[
    \beta_n=\beta \left(1 - \frac{1}{\log n}\right)\quad
    \text{and}
    \quad
    \beta=5/2.
\]

\begin{lemma}\label{lemma:super_sol}
There exist $N_0\in\zz{N}$ and $\kappa_0$ \textbf{independent
of $N_0$} such that for all $\kappa\geq \kappa_0$ and $n\geq N_0$,
\begin{equation}\label{eq:v_indc}
    v_{n+1} (x)\geq \zz{P}v_{n} (x)\left(1-\frac{1}{2} \zz{P}v_{n} (x)\right).
\end{equation}
\end{lemma}

The (rather technical) proof is deferred to section
\S\ref{ssec:super-solution} below.  (See \cite{bramson-cox-greven}
for a similar argument in the context of the KPP equation.) Given
Lemma~\ref{lemma:super_sol}, Proposition
\ref{proposition:uupperBound} is an easy consequence.

\begin{cor}\label{corollary:super-solution}
There exist $N_1\in\zz{N}$ and $\kappa>0$
such that for all $n\geq 0$,
\begin{equation}\label{eq:super-solution}
   u_n(x)\leq v_{N_1+n}(x)\leq 1.
\end{equation}
\end{cor}

\begin{proof}
Choose $N_1\geq N_0$ such that $\kappa:=  N_1\log N_1\geq \kappa_0$.
For such a choice of $(N_1,\kappa)$ we have
\[
  u_0(x)=\indic_{\{x=0 \}}\leq v_{N_1}(x)\leq 1.
\]
Moreover, by Lemma \ref{lemma:super_sol}, the function
$\tilde{v}_n(x):=v_{n+N_1}(x)$ satisfies \eqref{eq:v_indc}. The
conclusion now follows from the Comparison
Lemma~\ref{lemma:comparison}.
\end{proof}

\subsection{Representation of the conditional
distribution}\label{ssec:representationUn} Revesz \cite{Revesz96}
considers a branching random walk on $\zz{R}^{d}$ that is identical to
the branching random walk we have studied, \emph{except} that the
particle motion is by Gaussian $N (0,I)$ increments rather than
Uniform-$\mathcal{N}$ increments.  One of the main results of
{\Revesz}'s article asserts that, conditional on the event that there
is at least one particle of the $n$th generation in the ball $B$ of
radius $\varrho = \pi^{-1/2}$ centered at the origin, the expected
total number of such particles is of order $\Theta (\log n)$. His
argument seems to rest on the (unproven) assertion (see the first two
sentences of his \emph{Proof of Theorem 3}) that conditional on
the event that a region $C$ is occupied by at least one particle at
time $t$, the branching random walk consists of a single pinned random
walk off of which independent branching random walks are thrown.
There is no proof of this assertion (in fact, it is not even stated
clearly, as far as we can see).

We believe that Revesz' assertion is false.  The purpose of this
section is to give a representation related to that of Revesz' for the
conditional law of the occupation random variable $U_{n}(x)$ given the
event
\[
     G_{n,x}:=\{U_{n} (x)>0 \}.
\]
This representation is similar to  Revesz' in that it consists
of independent branching random walks thrown off a random path from
$(0,0)$ to $(n,x)$; however, the distribution of the random path is
\emph{not} that of a pinned simple random walk, but rather that of a
$u-$\emph{transformed} simple random walk. This is defined as
follows:

\begin{dfn}\label{definition:uSRW}
For each site $x$ and integer $n\geq 1$ such that $u_{n} (x)>0$, the
$u-$\emph{transformed simple random walk with endpoint} $(n,x)$ is the
$n-$step, time-inhomogeneous Markov chain $\{X_{m} \}_{0\leq m\leq n}$
on $\zz{Z}^{d}$ with initial point $0$ and transition probabilities
\begin{equation}\label{eq:uTransformTPs}
    q_{m} (z,y):= P (X_{m}=y\,|\, X_{m-1}=z)=
           P_{1} (y-z) \frac{u_{n-m} (x-y)}{\zz{P}u_{n-m} (x-z)}.
\end{equation}
\end{dfn}

\begin{remark}\label{remark:doob}
Except in the trivial case $n=1$, a $u-${transformed} random walk is
\emph{not} a Doob $h-$process, because the hitting probability
function $u_{n} (x)$ is not space-time harmonic for the simple
random walk, by equation~\eqref{eq:kpp} above. But a pinned random
walk \emph{is} an $h-$process: In particular, the one-step
transition probabilities of a pinned random walk conditioned to end
at~$x_{n}$ are given by
\begin{equation}\label{eq:rnPinned}
    q^{*}_{m} (z,y)= P_{1} (y-z) \frac{P_{n-m} (x_{n}-y)}{P_{n-m+1} (x_{n}-z)}.
\end{equation}
Since the function $P_{n-m} (z,x_{n})$ is space-time harmonic, the
transition probabilities $q^{*}$ are not the same as those of the
$u-${transformed}  random walk.
\end{remark}

\begin{lemma}\label{lemma:uTrWellDefined}
If $u_{n} (x)>0$ then the $u-${transformed} simple random walk with
endpoint $(n,x)$ is well-defined, and with probability one ends at
$X_{n}=x$.
\end{lemma}

\begin{proof}
What must be shown is that the Markov chain with transition
probabilities \eqref{eq:uTransformTPs} will visit no states
$(m,z)$ at which the denominator $\zz{P}u_{n-m} (x-z)$ is zero.
This is accomplished by noting that as long as $X_{m-1}$ is at a
site $z$ such that $u_{n-m+1} (x-z)>0$, then by
Lemma~\ref{lemma:hittingProbEqn} the denominator $\zz{P}u_{n-m}
(x-z)>0$, and so there is at least one site $y$ among the nearest
neighbors of $z$ such that $u_{n-m} (x-y)>0$. By
\eqref{eq:uTransformTPs}, the next state $X_{m}$ will then be
chosen from among the nearest neighbors such that $u_{n-m}
(x-y)>0$. This proves that the Markov chain is well-defined.  The
path ends at $X_{n}=x$ because $0$ is the only site at which
$u_{0}>0$.
\end{proof}

Our representation of the conditional distribution of the random
variable $U_{n} (x)$ given the event $G_{n,x}$  requires four mutually
independent sequences of random variables:
\begin{itemize}
 \item [($U_{a}$)] $\{X_{m} \}_{0\leq m\leq n}$ is a $u-${transformed}
simple random walk with endpoint $(n,x)$;
 \item [($U_{b}$)] $\{B_{m}
(w) \}_{0\leq m<n; \,w\in \zz{Z}^{d}}$ are independent
Bernoulli$(\beta_{m} (w))$ random variables;
 \item [($U_{c}$)]
$\{U^{i}_{m} (y) \}_{i\geq 0}$ are independent copies of the branching
random walk $\{U_{m} (y) \}$; and
 \item [($U_{d}$)] $\{\xi_{i}
\}_{i\geq 0}$ are independent and uniformly distributed on
$\mathcal{N}$.
\end{itemize}
The Bernoulli parameters are
\begin{equation}\label{eq:UBernouliiPars}
    \beta_{m} (w) =\frac{1}{2-\zz{P}u_{n-m-1} (x-w)};
\end{equation}
note that for large values of $n-m$ the parameters $\beta_{m} (w)$ are
uniformly close to $1/2$, because $u_{n-m} (x-w)$ is bounded by the
probability that the branching random walk will survive for $n-m$
generations.

\begin{prop}\label{proposition:URepresentation}
Assume that the offspring distribution is double-or-nothing, and let
$x$ be a site for which $u_{n} (x)>0$. Then
\begin{equation}\label{eq:URepresentation}
\aligned
    &\mathcal{L}\left( U_{n} (x)\,\bigg| \, U_{n} (x)\geq 1\right)\\
     =& \mathcal{L} \left(1+ \sum_{m=0}^{n-1} B_{m} (X_{m}) U^{m}_{n-m-1}
          (x-X_{m}-\xi_{m+1})\right).
\endaligned
\end{equation}
\end{prop}

\begin{proof}
The assertion~\eqref{eq:URepresentation} is equivalent to the
assertion (Claim~\ref{claim:uRep} below) that the conditional
distribution can be simulated by the following \emph{Method A}: (1)
Let a particle $\zeta$ execute a $u-${transformed} simple random walk
$\{X_{m} \}_{m\leq n}$ with endpoint $(n,x)$. (2) At each location
$(m,X_{m})$, where $0\leq m<n$, toss a $\beta_{m} (X_{m})-$coin to
determine whether or not to attach a descendant branching random
walk. (3) On the event that the coin toss is a Head, create a new
particle $\zeta_{m}$, let it make one jump $\xi_{m+1}$ to a
neighboring site, and then attach an independent branching random walk
starting from this new location. (4) Count the total number of
particles, including $\zeta$, that land at site $x$ at time~$n$.

\begin{claim}\label{claim:uRep}
This simulates the conditional distribution of the total number of
particles at site $x$ in generation~$n$ given the event $\{U_{n}
(x)\geq 1\}$.
\end{claim}

This claim is proved by induction on $n$. The case $n=1$ is
routine, but for the reader's convenience we shall present the
argument in detail.  First, the only sites $x$ such that $u_{1}
(x)>0$ are the nearest neighbors of the origin, so we assume that
$x$ is one of these five points.  Since $u_{0}=\delta_{0}$ is the
Kronecker delta function, $\zz{P}u_{0} (x)=1/5$, and so $\beta_{0}
(0)=1/ (2-1/5)=5/9$.  Now consider the first generation
$\mathcal{Z}_{1}$ of the branching random walk: this will be empty
unless the initial particle fissions, in which case the two
offspring are located at randomly chosen nearest neighbors of the
origin.  Consequently, the unconditional distribution of $U_{1}
(x)$ is
\begin{align*}
    P\{ U_{1} (x) =0\}&=\frac{1}{2} +\frac{1}{2} \times \frac{4}{5}\times \frac{4}{5};\\
    P\{ U_{1} (x) =1\}&=\frac{1}{2} \times 2\times \frac{4}{5}\times \frac{1}{5};\\
    P\{ U_{1} (x) =2\}&=\frac{1}{2} \times \frac{1}{5}\times \frac{1}{5}.
\end{align*}
It follows that the \emph{conditional} distribution of $U_{1} (x)$
given the event $\{U_{1} (x) > 0\}$ is that of $1$ plus a
Bernoulli($1/9$) random variable. This coincides with the
distribution of the random variable produced by Method A, because
$B_{0} (0)=~1$ with probability $\beta_{0} (0)= 5/9$, and on this
event the particle jumps to $x$ with probability $1/5$, leaving a
second particle at $x$.

Next, consider the branching random walk conditioned to have at least
one particle at site $x$ in generation $n\geq 2$. The first generation
must consist of two particles, at least one of which produces a
descendant branching random walk that places particles at $x$ in its
$(n-1)$st generation. Conditional on the event that two particles are
produced by the initial particle (that is, the event $\{Z_{1}=2 \}$),
each will have chance $p:=\zz{P}u_{n-1} (x)$ of producing a
descendant at site $x$ in generation $n$; consequently, the
conditional probability that \emph{both} particles will do so, given
that \emph{at least one} does, is
\[
    \frac{p^{2}}{p^{2}+ 2p(1-p)}=p\beta_{n-1} (0).
\]
Moreover, given that either one of the particles  produces a
particle at site~$x$ in generation $n$, the conditional probability
that its first jump is to site $y\in \mathcal{N}$ is
\begin{equation}\label{eq:utrStep1}
    P_{1} (y)\frac{u_{n-1} (x-y)}{\zz{P}u_{n-1} (x)};
\end{equation}
this is the distribution of the first step of a $u-${transformed}
random walk with endpoint $(n,x)$. Thus, a version of the random
variable $U_{n} (x)$, conditional on $\{U_{n} (x) \geq 1\}$, can be
produced by the following two-step procedure:

\medskip
(1) Place a particle
$\eta$ at a randomly chosen neighbor $y$ of $0$ according to the
distribution~\eqref{eq:utrStep1}, and attach to it a branching random
walk conditioned to produce at least one descendant at site $x-y$ in
its $(n-1)$st generation. By the induction hypothesis,
the contribution of offspring of $\eta$ to site $x$ in generation $n$
will be
\begin{equation}\label{eq:induction1}
    1+\sum_{m=1}^{n-1} B_{m} (X_{m}) U^{m}_{n-m-1}
          (x-X_{m}-\xi_{m+1}).
\end{equation}

 (2) With probability $p\beta_{n-1} (0)$, do the same with a second
particle $\tau$. Observe that, conditional on the event that this
second particle $\tau$ is attached, the contribution to site $x$ in
the $n$th generation will have distribution
\[
    \mathcal{L}\left(U_{n-1}(x-X_1)\,|\, U_{n-1}(x-X_1)>0\right).
\]
Since the particle $\tau$ is attached with probability $p\beta_{n-1}
(0)$, where $p$ is the probability that a particle
born at time $0$ will put a descendant at site $x$ in generation $n$,
step (2) has the same effect as this alternative: (2') With
probability $\beta_{n-1} (0)$, place a second particle~$\tau$ at a
randomly chosen (that is, uniformly distributed) neighbor $y$ of $0$,
and attach an independent copy of the branching random walk.  This,
together with the representation \eqref{eq:induction1} of the number
of offspring of $\eta$ at $x$ in generation $n$, shows that the total
number of particles at $x$ in generation $n$ will be
\begin{equation}\label{eq:induction2}
    1+\sum_{m=0}^{n-1} B_{m} (X_{m}) U^{m}_{n-m-1}
          (x-X_{m}-\xi_{m+1}),
\end{equation}
as desired. This completes the induction argument, and thus proves
\eqref{eq:URepresentation}.
\end{proof}

\subsection{Proof of
Lemma~\ref{lemma:super_sol}}\label{ssec:super-solution}
Let $x= (x_{1},x_{2})$; then
$$\zz{P}v_n(x)
   = v_n(x)\cdot e^{-\beta_n/(2n)}\cdot
   \frac{1}{5} w_{n} (x)
$$
where
\[
   w_n(x) = \left( e^{\beta_n /(2n)}+e^{-\beta_n x_1/n} + e^{\beta_n x_1/n} +e^{-\beta_n x_2/n}+e^{\beta_n
   x_2/n}\right).
\]
Then
\begin{multline*}
v_{n+1}(x) - \zz{P}v_n(x) +\frac{1}{2}(\zz{P}v_n(x))^2=\\
   v_n(x)e^{-\beta_n/(2n)}\frac{1}{5}\left(5e^{\beta_n/(2n)}\frac{n\log
       n}{(n+1)\log (n+1)}\exp\left(\frac{\theta_n|x|^2}{2}\right) - w_n(x)\right.\\
       \hskip3cm\left. + \frac{e^{-\beta_n/(2n)}}{10} v_n(x) w_n(x)^2
        \right),
\end{multline*}
where
\[
   \theta_n = \frac{\beta_n}{n}-\frac{\beta_{n+1}}{n+1}.
\]
Therefore it suffices to show that there exist $N_0$ and
$\kappa_0$ independent of $N_0$ such that for all $\kappa\geq
\kappa_0$ and for all $n\geq N_0$, the following holds:
\begin{equation}\label{eq:v_indc_2}
\aligned
       &5e^{\beta_n/(2n)}\frac{n\log n}{(n+1)\log (n+1)}\exp\left(\frac{\theta_n|x|^2}{2}\right) - w_n(x)\\
       &\quad + \frac{e^{-\beta_n/(2n)}}{10} v_n(x) w_n(x)^2
       \geq 0.
\endaligned
\end{equation}

\smallskip\noindent\textbf{a. Estimate of $e^{\beta_n/(2n)}\frac{n\log
       n}{(n+1)\log (n+1)}$}: First,
\[\aligned
    \frac{n\log n}{(n+1)\log (n+1)}
    &=1-\frac{(n+1)\log (n+1) - n\log n}{(n+1)\log (n+1)}\\
    &=1- \frac{1}{n+1}-\frac{n\log(1+1/n)}{(n+1)\log (n+1)}\\
    &= 1-\frac{1}{n+1}-\frac{1}{(n+1)\log (n+1)}+o\left(\frac{1}{n^2}\right)\\
    &= 1-\frac{1}{n}-\frac{1}{n\log n}+ \frac{1}{n^2}+o\left(\frac{1}{n^2}\right).
\endaligned
\]
Therefore, recall that $\beta_n = \beta(1-1/\log n)$,
\[\aligned
  &e^{\beta_n /(2n)} \frac{n\log n}{(n+1)\log (n+1)}\\
  = &\left(1 + \frac{\beta}{2n}-\frac{\beta}{2n\log n}+\frac{\beta_n^2}{8n^2} + O\left(\frac{1}{n^3}\right)\right)
   \cdot\left( 1-\frac{1}{n}-\frac{1}{n\log n}+ \frac{1}{n^2}+o\left(\frac{1}{n^2}\right)\right)\\
  = &1+ \frac{(\beta - 2)}{2n} -\frac{\beta+2}{2n\log
  n}+ \frac{\beta_n^2 - 4\beta+8}{8n^2}+o\left(\frac{1}{n^2}\right).
\endaligned
\]
Since $\beta_n\to \beta=5/2$, there exists $N_0\in\zz{N}$ such
that for all $n\geq N_0$,
\begin{equation}\label{eq:coef}
  e^{\beta_n /(2n)} \frac{n\log n}{(n+1)\log (n+1)}\geq 1+ \frac{(\beta - 2)}{2n} -\frac{\beta+2}{2n\log
  n}+ \frac{2}{8n^2}\geq 1.
\end{equation}

\smallskip\noindent\textbf{b. Estimate of $\theta_n$}:
Since $\beta_n = \beta(1 - 1/\log n)$, we have
\[\aligned
   \theta_n& = \beta\left(\frac{1-1/\log n}{n} -
   \frac{1-1/\log(n+1)}{n+1}\right)\\
           &=\beta\frac{1-(n+1)/\log n + n/\log(n+1)}{n(n+1)}.
\endaligned
\]
However,
\[\aligned
  \frac{n+1}{\log n} -\frac{n}{\log(n+1)} &= \frac{\log (n+1)+n\log(1+1/n)}{\log n \log
  (n+1)}\\
  &=\frac{1}{\log n} + O\left(\frac{1}{\log n \log (n+1)}\right),
\endaligned
\]
so it follows that
\begin{equation}\label{eq:theta_n}
  \theta_n =\beta\frac{1-1/\log n + O(1/(\log n \log
  (n+1)))}{n(n+1)}.
\end{equation}
\begin{claim} Enlarging $N_0$ if necessary, we have that for all $n\geq
N_0$,
\begin{equation}\label{eq:theta_beta}
   \beta \theta_n - \frac{\beta_n^2}{n^2}\geq \frac{1}{n^2\log n}.
\end{equation}
\end{claim}

\begin{proof}
[Proof of the claim]
Since $\beta_n = \beta(1-1/\log n)$,
\[\aligned
   &n^2\cdot \left(\beta \theta_n - \frac{\beta_n^2}{n^2}\right)\\
   =&\beta^2\left\{\left(1-\frac{1}{n+1}\right)\left(1-\frac{1}{\log n} + O\left(\frac{1}{\log n \log
  (n+1)}\right)\right)\right.\\
   &\hskip3cm\left. - \left(1 - \frac{2}{\log n} + \frac{1}{(\log
  n)^2}\right)\right\}\\
  =&\beta^2 \left(\frac{1}{\log n} + o\left(\frac{1}{\log n}\right)
  \right).
\endaligned
\]
The relation \eqref{eq:theta_beta} follows since $\beta = 5/2>1$.
\end{proof}

\smallskip\noindent\textbf{c. Proof of \eqref{eq:v_indc_2} for $|x|\geq
3n$}: \eqref{eq:theta_beta} implies, enlarging $N_0$ if necessary,
that for all $n\geq N_0$, $\theta_n  \geq 2/n^2$. Hence when
$|x|\geq 3n$,
\[
   \theta_n |x|^2/2 \geq \beta_n |x_i|/n,\quad i=1,2,
\]
and
\[
  5\exp\left(\frac{\theta_n|x|^2}{2}\right)\geq w_n(x).
\]
The relation \eqref{eq:v_indc_2} follows by noting \eqref{eq:coef}.

\smallskip\noindent\textbf{d. Estimate of $w_n(x)$}:
For $|x|/n$  sufficiently small,  Taylor expansion yields
\begin{equation}\label{eq:w_n}\aligned
  w_n(x)=&e^{\beta_n /(2n)}+(e^{-\beta_n x_1/n} + e^{\beta_n x_1/n}) + (e^{-\beta_n
   x_2/n}+e^{\beta_n x_2/n})\\
  = &1 + \frac{\beta}{2n} -\frac{\beta}{2n\log n} + \frac{\beta_n^2}{8n^2}+O\left(\frac{1}{n^3}\right)\\
   &+2 + \frac{\beta_n^2 x_1^2}{n^2} + \frac{\beta_n^4 x_1^4}{12 n^4} + O\left(\left(\frac{x_1}{n}\right)^6\right) \\
  & + 2 +\frac{\beta_n^2 x_2^2}{n^2} + \frac{\beta_n^4 x_2^4}{12 n^4} +
  O\left(\left(\frac{x_2}{n}\right)^6\right)\\
  = &5 +\left[  \frac{\beta}{2n}-\frac{\beta}{2n\log
  n}\right]+\left[\frac{\beta_n^2}{8n^2}+O\left(\frac{1}{n^3}\right)\right]\\
   &\quad+ \frac{\beta_n^2 |x|^2}{n^2} + \left[\frac{\beta_n^4 (x_1^4+ x_2^4)}{12 n^4}
  +  O\left(\frac{|x|^6}{n^6}\right)\right],
\endaligned
\end{equation}

\smallskip\noindent\textbf{e. Estimate of $e^{\beta_n /(2n)}\cdot \frac{n\log n}{(n+1)\log (n+1)}\cdot\exp(\theta_n
   |x|^2/2)$}: By \eqref{eq:coef}, for all $n\geq N_0$,
\begin{equation}\label{eq:fist_term}\aligned
   & e^{\beta_n /(2n)}\cdot \frac{n\log n}{(n+1)\log (n+1)}\cdot\exp\left(\frac{\theta_n
   |x|^2}{2}\right)\\
  \geq & \left(1+ \frac{\beta - 2}{2n} -\frac{\beta+2}{2n\log
  n}+ \frac{2}{8n^2}\right)\cdot \left( 1 +\frac{\theta_n
   |x|^2}{2} + \frac{\theta_n^2
   |x|^4}{8}\right)\\
  \geq &1 +  \left[\frac{\beta - 2}{2n} -\frac{\beta+2}{2n\log
  n}\right]+ \frac{2}{8n^2}+\frac{\theta_n
   |x|^2}{2} + \frac{  \theta_n^2
   |x|^4}{8}.
\endaligned
\end{equation}

\smallskip \noindent \textbf{f: Their difference}: By
\eqref{eq:fist_term} and \eqref{eq:w_n},
\begin{equation}\label{eq:diff}
\aligned
   &5e^{\beta_n /(2n)}\cdot \frac{n\log n}{(n+1)\log (n+1)}\cdot\exp\left(\frac{\theta_n
   |x|^2}{2}\right) - w_n(x)\\
   \geq& 5 \left[\frac{\beta - 2}{2n} -\frac{\beta+2}{2n\log
  n}\right] - \left[\frac{\beta}{2n}-\frac{\beta}{2n\log
  n}\right]\\
  \quad &+\frac{10}{8n^2}-  \frac{\beta_n^2}{8n^2} +O\left(\frac{1}{n^3}\right)\\
  \quad & + \left(\beta\theta_n
   - \frac{\beta_n^2}{n^2}\right) |x|^2 \\
   \quad &+ \frac{5  \theta_n^2
   |x|^4}{8}-  \frac{\beta_n^4 (x_1^4+ x_2^4)}{12 n^4} +
  O\left(\frac{|x|^6}{n^6}\right).
\endaligned
\end{equation}
Since $\beta = 5/2$,
\begin{equation}\label{eq:diff_1st}
   5 \left(\frac{\beta - 2}{2n} -\frac{\beta+2}{2n\log
  n}\right) - \left(\frac{\beta}{2n}-\frac{\beta}{2n\log
  n}\right)=-\frac{10}{n\log n},
\end{equation}
and enlarging $N_0$ if necessary we can assume that for all $n\geq
N_0$,
\begin{equation}\label{eq:diff_2nd}
  \frac{10}{8n^2}-  \frac{\beta_n^2}{8n^2}+O\left(\frac{1}{n^3}\right)\geq 0.
\end{equation}
Moreover, $\theta_n\sim \beta/n^2$, it follows that for all $n$
sufficiently large,
\begin{equation}\label{eq:4_moment}
   \frac{5  \theta_n^2
   |x|^4}{8}-  \frac{\beta_n^4 (x_1^4+ x_2^4)}{12 n^4}\geq \left(\frac{5  \theta_n^2
   }{8} - \frac{\beta_n^4 }{12 n^4}\right)\cdot |x|^4>\frac{|x|^4}{2n^4}.
\end{equation}

\smallskip\noindent\textbf{g. Proof of \eqref{eq:v_indc_2} for $\delta
n\geq |x|>\sqrt{10 n}$, where $\delta>0$ is sufficiently small}:
By \eqref{eq:theta_beta}, when $|x|>\sqrt{10 n}$,
\[
   \left(\beta\theta_n
    - \frac{\beta_n^2}{n^2}\right) |x|^2
   \geq \frac{10}{n\log n}.
\]
Hence, by \eqref{eq:diff}, \eqref{eq:diff_1st}, \eqref{eq:diff_2nd}
and \eqref{eq:4_moment}, the relation \eqref{eq:v_indc_2} holds for
$x$ such that $|x|>\sqrt{10 n}$ and $|x|/n$ sufficiently small.

\smallskip\noindent\textbf{h. Proof of \eqref{eq:v_indc_2} for $3n\geq |x|\geq \delta
n$:} By \eqref{eq:theta_beta}, for all $n\geq N_0$,
\[
  \theta_n \geq \frac{\beta_n^2}{n^2\beta} - \frac{1}{\beta n^2\log
  n}.
\]
Hence when $|x|\leq 3n$,
\[
  \exp\left(\frac{\theta_n|x|^2}{2}\right)
  \geq \frac{\exp\left(\frac{\beta_n^2|x|^2}{5n^2}\right)}{\exp\left(\frac{|x|^2}{5n^2\log n}\right)}
  \geq \exp\left(-\frac{2}{\log
  n}\right)\cdot\exp\left(\frac{\beta_n^2|x|^2}{5n^2}\right).
\]
By \eqref{eq:coef}, to show \eqref{eq:v_indc_2} it is sufficient
to show that for all $n$ sufficiently large,
\[\aligned
   &\left( e^{\beta_n /(2n)}+e^{-\beta_n x_1/n} + e^{\beta_n x_1/n} +e^{-\beta_n x_2/n}+e^{\beta_n
   x_2/n}\right)\\
   \leq&\, 5\exp\left(-\frac{2}{\log
  n}\right)\cdot\exp\left(\frac{\beta_n^2|x|^2}{5n^2}\right).
\endaligned
\]
Since $|x|\leq 3n$,
\[\aligned
  &\left(1-\exp\left(-\frac{2}{\log
  n}\right)\right)\cdot \exp\left(\frac{\beta_n^2|x|^2}{5n^2}\right)\\
  \leq& \left(1-\exp\left(-\frac{2}{\log
  n}\right)\right) \exp\left(\frac{9\beta_n^2}{5}\right) = o(1).
\endaligned
\]
Hence it  suffices to show that
\begin{equation}\label{eq:v_aty_mid}
\aligned
   &\liminf_n \inf_{3n\geq |x|\geq \delta n}
   \left\{ 5\exp\left(\frac{\beta_n^2|x|^2}{5n^2}\right) \right.\\
    &\hskip1cm\left.- \left(
1+e^{-\beta_n x_1/n} + e^{\beta_n x_1/n} +e^{-\beta_n
x_2/n}+e^{\beta_n
   x_2/n}\right) \right\} >0.
\endaligned
\end{equation}
By elementary calculus,
\[
   e^{-\beta_n x_1/n} + e^{\beta_n x_1/n} +e^{-\beta_n x_2/n}+e^{\beta_n
   x_2/n}
   \leq 2 + e^{-\beta_n |x|/n} + e^{\beta_n |x|/n},
\]
thus
\[\aligned
   &5\exp\left(\frac{\beta_n^2|x|^2}{5n^2}\right) - \left( 1+e^{-\beta_n x_1/n} + e^{\beta_n x_1/n} +e^{-\beta_n
     x_2/n}+e^{\beta_n    x_2/n}\right)\\
   \geq & 5\exp\left(\frac{\beta_n^2|x|^2}{5n^2}\right) - 3 - e^{-\beta_n |x|/n} - e^{\beta_n
   |x|/n}.
\endaligned
\]
Relation \eqref{eq:v_aty_mid} now follows from the simple fact that
\[
   f(x):= 5 e^{x^2/5} - 3 - e^x - e^{-x}
\]
is strictly increasing for $x\geq 0$ and equals 0 only when $x=0$.

\smallskip\noindent\textbf{i. Proof of \eqref{eq:v_indc_2} when
$|x|\leq\sqrt{10 n}$}: Since $|x|\leq\sqrt{10 n}=o(n)$, by
relations \eqref{eq:diff}, \eqref{eq:diff_1st},
\eqref{eq:diff_2nd}, \eqref{eq:theta_beta} and
\eqref{eq:4_moment},  we need only show that there exists
$\kappa_0$ such that if $\kappa\geq \kappa_0$ and  $n\geq N_0$,
then
\[
  \frac{e^{-\beta_n/(2n)}}{10} v_n(x) w_n(x)^2
  \geq \frac{10}{n\log n}.
\]
 Since $w_n(x)\geq 5$, when $|x|\leq\sqrt{10 n}$,
\[
   \frac{e^{-\beta_n/(2n)}}{10} v_n(x) w_n(x)^2\geq \frac{25\kappa}{10n\log n} \exp(-6\beta),
\]
so $\kappa_0$ can be chosen as $4\exp(6\beta)$, which is
independent of $N_0$.

\qed

\section*{Acknowledgments}

We thank Lenya Ryzhik for a helpful suggestion regarding the proof of
Proposition \ref{prop:mgf} and Andrej Zlatos for suggesting the form
of the super-solution \eqref{eq:v_n}. We also thank Michael Wichura
for carefully reading an earlier draft and pointing out a mistake in
our original proof of Lemma \ref{lemma:um0}. Finally, we are grateful
to a referee for some valuable suggestions regarding the exposition
and for alerting us to the work of \Revesz.

\end{document}